\documentclass[11pt,a4paper]{article}
\usepackage{mathrsfs}
\usepackage{amssymb,amsmath,amsfonts,amsthm}
\usepackage{amsfonts}
\usepackage{graphicx,subfigure,epstopdf}
\usepackage[usenames]{color}
\usepackage{url}
\usepackage{algorithm,algorithmic}
\usepackage{dsfont,verbatim}
\usepackage[colorlinks,linktocpage,linkcolor=blue]{hyperref}
\usepackage[margin=1.1in]{geometry}
\usepackage{enumerate,enumitem}
\usepackage[normalem]{ulem}
\usepackage{cite}
\usepackage{bm}
\allowdisplaybreaks

\definecolor{darkred}{rgb}{.7,0,0}

\definecolor{green}{rgb}{0,0.7,0}

\newtheoremstyle{thmm}{1.5ex plus 1ex minus .2ex}{1.5ex plus 1ex minus.2ex}{\rmfamily}{}{\bfseries}{}{1em}{} \theoremstyle{thmm}
\newtheorem{theorem}{Theorem}[section]
\newtheorem{lemma}{Lemma}[section]

\newtheorem{corollary}{Corollary}[section]

\newtheorem{remark}{Remark}[section]
\renewenvironment{proof}[1][Proof]{\noindent\textit{#1. }
}{\hfill$\square$}

\newcommand{\vertiii}[1]
{{\left\vert\kern-0.25ex\left
\vert\kern-0.25ex\left\vert #1
    \right\vert\kern-0.25ex\right
\vert\kern-0.25ex\right\vert}}

\def\d{{\rm d}}

\def\R{\mathbb{R}}
\def\C{\mathbb{C}}

\def\i{\mathrm{i}}
\def\ii{|\hspace{-1pt}\|}

\newcommand\blfootnote[1]{%
  \begingroup
  \renewcommand\thefootnote{}\footnote{#1}%
  \addtocounter{footnote}{-1}%
  \endgroup
}

\title{\Large\bf Maximal regularity of multistep fully discrete\\ finite element methods for parabolic equations}
\author{\normalsize Buyang Li\thanks{Department of Applied Mathematics, The Hong Kong Polytechnic University, Hung Hom, Kowloon, Hong Kong.
Email address: buyang.li@polyu.edu.hk }}

\date{}

\begin{document}

\maketitle

\vspace{-10pt}

\begin{abstract}
This article extends the semidiscrete maximal $L^p$-regularity results in \cite{Li2019} to multistep fully discrete finite element methods for parabolic equations with more general diffusion coefficients in $W^{1,d+\beta}$, where $d$ is the dimension of space and $\beta>0$. The maximal angles of $R$-boundedness are characterized for the analytic semigroup $e^{zA_h}$ and the resolvent operator $z(z-A_h)^{-1}$, respectively, associated to an elliptic finite element operator $A_h$. Maximal $L^p$-regularity, optimal $\ell^p(L^q)$ error estimate, and $\ell^p(W^{1,q})$ estimate are established for fully discrete finite element methods with multistep backward differentiation formula. \\ 

\end{abstract}

\blfootnote{This work is partially supported by a Hong Kong RGC grant (project no. 15301818). }

\vspace{-10pt}

\setlength\abovedisplayskip{5pt}
\setlength\belowdisplayskip{5pt}

\section{Introduction}
\setcounter{equation}{0}

Let $\varOmega\subset\R^d$, $d\in\{2,3\}$, be a polygonal or polyhedral domain. We consider the initial and boundary value problem for a linear parabolic partial differential equation (PDE)
\begin{align} 
\label{PDE1} 
\left\{\begin{array}{ll}
\displaystyle
\frac{\partial u(t,x)}{\partial t}- \sum_{i,j=1}^d\frac{\partial}{\partial x_i}\bigg(a_{ij}(x)\frac{\partial u(t,x)}{\partial x_j}\bigg) = f(t,x)
&\mbox{for}\,\,(t,x)\in \R_+\times \varOmega,\\[10pt]
u(t,x)=0 
&\mbox{for}\,\,(t,x)\in \R_+\times\partial\varOmega ,\\[7pt]
u(0,x)=u_0(x)
&\mbox{for}\,\,x\in \varOmega ,
\end{array}\right. 
\end{align} 
where $a_{ij}=a_{ji}$ are real-valued functions satisfying the following ellipticity and regularity conditions: 
\begin{align}
&\lambda^{-1}|\xi|^2 
\le \mbox{$\sum_{i,j=1}^d$} a_{ij}(x) \xi_i\xi_j 
\le \lambda|\xi|^2,
&&\forall\, \xi=(\xi_1,\dotsc,\xi_d)\in\R^d,\,\,\,\forall\, x\in\varOmega , 
\label{Ellipticitiy} \\[5pt]
&a_{ij}=a_{ji}\in W^{1,d+\beta} 
&&\mbox{for some constants $\lambda,\beta>0$}.
\label{Regularity:aij}
\end{align}

It is known that the elliptic partial differential operator $A=\sum_{i,j=1}^d\frac{\partial}{\partial x_i}\big(a_{ij}(x)\frac{\partial }{\partial x_j}\big)$, under the Dirichlet boundary condition, generates a bounded analytic semigroup on the space $L^q:=L^q(\varOmega)$ for all $1<q<\infty$; see \cite[Theorem 2.4]{Ouhabaz1995}. When $u_0=0$, the solution of \eqref{PDE1} possesses maximal $L^p$-regularity on $L^q$, namely, 
\begin{align}\label{MaxLpReg}
\hspace{-5pt}
&\|\partial_tu\|_{L^p(\R_+;L^q)}
+\|A u\|_{L^p(\R_+;L^q)}\leq C_{p,q}
\|f\|_{L^p(\R_+;L^q)} 
\quad\forall\, 1<p,q<\infty ,\,\,\, 
\end{align} 
which is an important tool in studying well-posedness and regularity of solutions to nonlinear parabolic PDEs; see \cite{Amann1995,ClementPruss1992,LiYang2015}. In the numerical solution of parabolic equations, it is also desirable to have a discrete analogue of this estimate, which has a number of applications in analysis of stability and convergence of numerical methods for nonlinear parabolic problems, including semilinear parabolic equations with strong nonlinearities \cite{Geissert2007}, quasi-linear parabolic equations with nonsmooth coefficients \cite{LiSun2015-regularity}, optimal control and inverse problems \cite{Leykekhman-Vexler-2016,Leykekhman-Vexler-Walter-2019}, and so on.

Let $S_h$, $0<h<h_0$, be a family of Lagrange finite element subspaces of $H^1_0(\varOmega)$ consisting of all piecewise polynomials of degree$\,\le r$ subject to a quasi-uniform triangulation of the domain $\varOmega$, where $r\ge 1$ is any given integer. It is known that the semidiscrete finite element solutions of \eqref{PDE1}, defined by  
\begin{align}\label{FEMEq0}
\left\{\begin{aligned} 
&(\partial_t u_h,v_h)+\sum_{i,j=1}^d (a_{ij} \partial_ju_h ,\partial_iv_h)=(f,v_h) 
\quad \forall\, v_h\in S_h ,\,\,
\forall\, t > 0 ,\\[3pt]
&u_h(0)=u_{h,0} ,
\end{aligned} 
\right.
\end{align} 
possesses maximal $L^p$-regularity similarly as the continuous problem. Namely, when $u_{h,0}=0$ the solution of \eqref{FEMEq0} satisfies 
\begin{align}
&\|\partial_tu_h\|_{L^p(\R_+;L^q)}
+\|A_hu_h\|_{L^p(\R_+;L^q)}\leq C_{p,q} 
\|f\|_{L^p(\R_+;L^q)}  
\quad \forall\, 1<p,q<\infty,
\label{MaxLpReg-FEM}
\end{align} 
with a constant $C_{p,q}$ independent of the mesh size $h$, where $A_h:S_h\rightarrow S_h$ is the finite element approximation of the elliptic operator $A$, defined by   
\begin{align}\label{Discrete-Laplacian}
(A_h\phi_h,\varphi_h)=-\sum_{i,j=1}^d (a_{ij}\partial_j \phi_h,\partial_i\varphi_h) \quad \forall\, \phi_h, \varphi_h\in S_h .
\end{align} 
Estimate \eqref{MaxLpReg-FEM} was proved in smooth domains with the Neumann boundary condition \cite{Geissert2006,Li2015,Kashiwabara-Kemmochi-2020} and in polyhedral domains with the Dirichlet boundary condition \cite{Li2019,LiSun2017-MCOM}. The proof of \eqref{MaxLpReg-FEM} is closely related to the proof of the following maximum-norm stability  (cf. \cite{Leykekhman2004,NitscheWheeler1982,SchatzThomeeWahlbin1980,SchatzThomeeWahlbin1998,ThomeeWahlbin2000}):
\begin{align}\label{maximum-norm}
\|u-u_h\|_{L^\infty(0,T;L^\infty)}
\le C\ell_h \inf_{\chi_h\in \in L^\infty(0,T;S_h)}\|u-\chi_h\|_{L^\infty(0,T;L^\infty)} ,
\end{align} 
with $\ell_h=\ln(2+1/h)$. 
The two results \eqref{MaxLpReg-FEM} and \eqref{maximum-norm} are often proved simultaneously by similar techniques. 

The extension of semidiscrete maximal $L^p$-regularity results to fully discrete FEMs has been established for several different time discretization methods, including the backward Euler method \cite{AshyralyevPiskarevWeis2002,LiSun2017-SINUM}, discontinuous Galerkin method \cite{LeykekhmanVexler2016-NM}, $\theta$-schemes \cite{KemmochiSaito}, and A-stable multistep and Runge--Kutta methods \cite{KovacsLiLubich2016}. All these proofs use A-stability of the methods. It is known that A-stable multistep methods can be at most second-order accurate (cf. \cite[p. 247, Theorem 1.4]{Hairer-Wanner-1996}). Hence, maximal $L^p$-regularity of fully discrete FEMs with higher-order multistep methods, including backward differentiation formulae (BDF), has not been proved so far.  

The BDF method is the one of the most popular high-order methods in solving parabolic equations due to its ease of implementation. 
For $k=1,\dots,6$, we denote by $u^n$ and $u_h^n$, $n=k,\dots,N$, the solution of the semidiscrete and fully discrete $k$-step BDF methods, given by 
\begin{align}\label{BDF-Laplacian-1}
\frac{1}{\tau} \sum_{j=0}^k \delta_j u^{n-j}
=A u^n+f^n, \quad n \geq k 
\end{align}
and
\begin{align}\label{BDF-FEM}
\frac{1}{\tau} \sum_{j=0}^k \delta_j u_h^{n-j}
=A_h u_h^n+f_h^n, \quad n \geq k,
\end{align}
respectively, 
where the starting values $u^n$ and $u_h^n$, $n=0,\dots,k-1$, in the $k$-step BDF method are assumed to be given (obtained by other methods), and $\delta_j$, $j=0,\dots,k$, are the coefficients in the polynomial 
\begin{equation}
\label{BDF-delta-j}
\delta (\zeta) {} := \sum_{j=1}^k \frac 1j  (1-\zeta)^j 
=\sum\limits^k_{j=0}\delta_j \zeta ^{j} . 
\end{equation}
It is known that the $k$-step BDF method is A($\alpha_k$)-stable with angles $\alpha_1=\alpha_2=0.5\pi$, $\alpha_3=0.478 \pi$, $\alpha_4=0.408 \pi$, $\alpha_5=0.288 \pi$ and $\alpha_6=0.099 \pi$; see \cite[Section V.2]{Hairer-Wanner-1996}. Here, A$(\alpha)$-stability is equivalent to 
$$
|\arg \delta (\zeta)|\le \pi-\alpha \quad\mbox{for}\,\,\,\zeta\in\C\,\,\,\mbox{such that}\,\,\,|\zeta|\le 1,
$$ 
and A-stability is equivalent to A$(\alpha)$-stability with $\alpha=0.5\pi$. In particular, only one-step and two-step BDF methods are A-stable. 

Maximal $L^p$-regularity of the $k$-step BDF method, with $k=3,\dots,6$, requires an additional angle condition, i.e,  
\begin{equation}\label{angle-condition}
\mbox{$\big\{z(z -A)^{-1} \,:\, z\in\Sigma_{\theta+\frac{\pi}{2}} \big\}$ is $R$-bounded on $L^q$ for some angle $\theta>\frac{\pi}{2}-\alpha_k$,}
\end{equation} 
where $\Sigma_\theta$ is a sector on the complex plane defined by 
$$\Sigma_\theta
=\{z\in\C\backslash\{0\}: |{\rm arg}(z)|<\theta\} ,$$ 
and $R$-boundedness is a stronger concept than boundedness required in the context of maximal $L^p$-regularity (see the rigorous definition in the next section). For the continuous operator $A$, condition \eqref{angle-condition}  holds for all $\theta\in(0,\frac{\pi}{2})$ and therefore, maximal $L^p$-regularity of semidiscretization in time holds for all BDF methods up to order $6$. These results were proved in \cite[Theorems 4.1--4.2 and Remark 4.2]{KovacsLiLubich2016}. 

Similarly, the extension of maximal $L^p$-regularity to fully discrete FEMs, with the $k$-step BDF method in time, requires the following condition: 
\begin{equation}\label{angle-condition2} 
\mbox{$\big\{z(z -A_h)^{-1}P_h \,:\, z\in\Sigma_{\theta+\frac{\pi}{2}} \big\}$ is $R$-bounded on $L^q$ for some angle $\theta>\frac{\pi}{2}-\alpha_k$,} 
\end{equation} 
with an $R$-bound independent of $h$, where $P_h$ denotes the $L^2$-orthogonal projection onto the finite element space. However, \eqref{angle-condition2} has only been proved for sufficiently small $\theta$ (for example, see \cite[Corollary 2.1]{Li2019} and \cite{LiSun2017-MCOM}) for parabolic equations with constant coefficients. From the proof therein, it is not clear whether $\theta>\frac{\pi}{2}-\alpha_k$ for the $k$-step BDF method with $k=3,\dots,6$, and how much regularity the diffusion coefficients $a_{ij}$ should possess for \eqref{angle-condition2} to be valid in general polygons and polyhedra. 

In this article, we fill in the gap between semidiscrete FEMs and multistep fully discrete FEMs. In particular, we prove that \eqref{angle-condition2} holds for all $\theta\in(0,\frac{\pi}{2})$, for diffusion coefficients $a_{ij}\in W^{1,d+\beta}(\varOmega)$ with $\beta>0$. As a result, we obtain discrete maximal $L^p$-regularity of fully discrete FEMs with general $k$-step BDF methods. Namely, with zero starting values $u_h^n=0$, $n=0,\dots,k-1$, the solution of \eqref{BDF-FEM} satisfies 
\begin{equation*}
        \big\|(d_\tau u_h^n)_{n=k}^N\big\|_{\ell^p(L^q)} + \big\|(A_h  u_h^n )_{n=k}^N\big\|_{\ell^p(L^q)} \leq C_{p,q}\big\|(f_h^n)_{n=k}^N\big\|_{\ell^p(L^q)}
        \qquad\quad\,\,\,\forall\, 1<p,q<\infty ,
\end{equation*}
with $d_\tau u^n:= (u^n-u^{n-1})/\tau$. In addition, we establish an error estimate between semidiscrete and fully discrete solutions in the temporally discrete $L^p(0,T;L^q)$ norm, for $1<p,q<\infty$: 
\begin{align*}
\!\!\!  \big\|(P_hu^n-u_h^n )_{n=k}^N\big\|_{\ell^p(L^q)} 
        \le C_{p,q}\big\|(P_h u^n - R_hu^n )_{n=k}^N\big\|_{\ell^p(L^q)} 
      +C_{p,q} \sum_{n=0}^{k-1}\big\|P_h u^n - u_h^n \big\|_{L^q} ,
\end{align*}
where $R_h$ denotes the Ritz projection onto the finite element space. 
In the case $p=q=\infty$, we obtain a fully discrete analogue of \eqref{maximum-norm}, i.e., 
\begin{align*}
\max_{1\le n\le N} \|P_hu^n-u_h^n\|_{L^\infty} 
\le C\ell_{N} \max_{1\le n\le N} \|P_h u^n - R_hu^n\|_{L^\infty} 
+ C\max_{1\le n\le k-1} \big\|P_h u^n - u_h^n \big\|_{L^\infty} ,
\end{align*}
with $\ell_{N}= \ln (1+N) $. 
The three results above hold in general polygons and polyhedra, with constants independent of $h$, $\tau$ and $N$. 
In convex polygons and polyhedra, we furthermore obtain the following temporally discrete $L^p(0,T;W^{1,q})$ estimate (for zero starting values $u_h^n=0$, $n=0,\dots,k-1$): 
\begin{equation*}
        \big\|(d_\tau u_h^n )_{n=k}^N\big\|_{\ell^p(W^{-1,q})}   
        +\big\|(u_h^n )_{n=k}^N\big\|_{\ell^p(W^{1,q})}   
        \le C_{p,q}\big\|(f_h^n )_{n=k}^N\big\|_{\ell^p(W^{-1,q})}  
\quad\forall\, 1<p,q<\infty .
\end{equation*}

The rest of this paper is organized as follows. In the next section, we introduce basic notations and preliminary results to be used in this paper. The main theorems are presented in Section \ref{section:main} with proofs for most of the results. The most technical proof is postponed to Section \ref{section:proof}.

%
%
%
%
%
%

\section{Preliminaries}

We denote by $\varOmega\subset\R^d$, with $d\in\{2,3\}$, a general polygon or polyhedron (possibly nonconvex) unless otherwise stated. We denote by $W^{s,q}(\varOmega)$ the conventional Sobolev space for $s\in\R$ and $1\le q\le \infty$, with the abbreviations 
$$
W^{s,q}=W^{s,q}(\varOmega),
\quad
L^{q}=W^{0,q}(\varOmega)
\quad\mbox{and}\quad 
H^{s}=W^{s,2}(\varOmega)
.
$$ 
The space of infinitely smooth functions with compact support in $\varOmega$ is denoted by $C^\infty_0$, and the completion of $C^\infty_0$ in $H^s$ is denoted by $H^{s}_0$. The space of H\"older continuous functions on $\overline\varOmega$ is denoted by $C^{\beta}$. 

For any $q\in[1,\infty]$, we denote by $q'\in[1,\infty]$ be the number satisfying 
$1/q+1/q'=1$, 
and denote by 
\begin{align}\label{def-inner-product}
(u,v)=\int_\varOmega u(x)v(x)\d x 
\quad 
\forall\, u\in L^q\,\,\,\mbox{and}\,\,\,v\in L^{q'},
\end{align}
the pairing between two real-valued functions on $\varOmega$. 

For a sequence $u^n$, $n=0,1,\dots,$ we denote 
\begin{align*}
\dot u^n:= \frac{1}{\tau} \sum_{j=0}^k \delta_j u^{n-j}\,\,\,\mbox{for}\,\,\,n\ge k,
\quad\mbox{and}\quad
d_\tau u^n:=\frac{u^n-u^{n-1}}{\tau} \,\,\,\mbox{for}\,\,\,n\ge 1. 
\end{align*}

For any Banach space $X$ and any sequence $v^n\in X$, $n=1,\dots,N$, we denote 
\begin{align*}
\|(v^n)_{n=1}^N\|_{\ell^p(X)}
:=
\left\{
\begin{aligned}
&\bigg(\tau \sum_{n=1}^N \|v^n\|_{X}^p\bigg)^{\frac1p} &&\mbox{for}\,\,\,1\le p<\infty,\\[3pt]
&\sup_{1\le n\le N}\|v^n\|_{X} &&\mbox{for}\,\,\,p=\infty .
\end{aligned}
\right.
\end{align*}

For any given index set $D$, a collection of operators $\{M(z):L^q\rightarrow L^q: z\in D\}$ is called {\it $R$-bounded} on $L^q$, with $1<q<\infty$, if and only if the following inequality holds for all finite subcollection of operators $M(z_1),...,M(z_m)$: 
\begin{align*}
\bigg\|\bigg(\sum_{j=1}^m
|M(z_j)v_j|^2\bigg)^{\frac{1}{2}}\bigg\|_{L^q}
\! \! \! \leq C_R\bigg\|\bigg(\sum_{j=1}^m
|v_j|^2\bigg)^{\frac{1}{2}}\bigg\|_{L^q} ,
\quad
\forall\,\, v_1,...,v_m\in L^q ,
\end{align*}
where the smallest constant $C_R$ satisfying this inequality is called the $R$-bound of the collection. This characterization of $R$-boundedness on $L^q$ is equivalent to the definition of $R$-boundedness of operators on a general Banach space; see \cite[Section 1.f]{Weis2001-1}.  


It is known that maximal $L^p$-regularity of the semidiscrete BDF method \eqref{BDF-Laplacian-1} is related to the $R$-boundedness of the resolvent operators $z (z - A)^{-1}$, as shown in the following theorem.

\begin{theorem}[{\hspace{-0.25pt}\cite[Theorems 4.1--4.2]{KovacsLiLubich2016}}]\label{THM:BDFk}
{\it
If the collection of operators
$\big\{ z (z - A)^{-1} \,:\, z\in \Sigma_{\theta+\frac{\pi}{2}}\big\}$ is $R$-bounded on $L^q$ for an angle $\theta>\frac{\pi}{2}-\alpha_k$, then the semidiscrete solution given by \eqref{BDF-Laplacian-1}, with zero starting values $u^n=0$ for $n=0,\dots,k-1$, satisfies the temporally discrete maximal $L^p$-regularity estimate: 
    \begin{equation*}
        \big\|( \dot u^n )_{n=k}^N\big\|_{\ell^p(L^q)} + \big\|(A u^n )_{n=k}^N\big\|_{\ell^p(L^q)} \leq C_{p,q}\big\|(f^n)_{n=k}^N\big\|_{\ell^p(L^q)} 
        \quad\forall\, 1<p,q<\infty ,
    \end{equation*}
where the constant $C_{p,q}$ is independent of $\tau$ and $N$.
}
\end{theorem}

Similarly, maximal $L^p$-regularity of fully discrete FEMs with BDF methods in time is summarized in the following theorem.

\begin{theorem}[{\hspace{-0.25pt}\cite[Theorems 6.1]{KovacsLiLubich2016}}]\label{THM:FEM-BDFk}
{\it
If the collection of operators
$\big\{ z (z - A_h)^{-1} P_h \,:\, z\in \Sigma_{\theta+\frac{\pi}{2}} \big\}$ is $R$-bounded on $L^q$ for an angle $\theta>\frac{\pi}{2}-\alpha_k$ (with an $R$-bound independent of $h$), then the fully discrete finite element solution given by \eqref{BDF-FEM}, with zero starting values $u_h^n=0$ for $n=0,\dots,k-1$, satisfies the following discrete maximal $L^p$-regularity estimate: 
    \begin{equation*}
        \big\|( \dot u_h^n )_{n=k}^N\big\|_{\ell^p(L^q)} + \big\|(A_hu_h^n )_{n=k}^N\big\|_{\ell^p(L^q)} \leq C_{p,q}\big\|(f_h^n)_{n=k}^N\big\|_{\ell^p(L^q)} 
        \quad\forall\, 1<p,q<\infty ,
    \end{equation*}
where the constant $C_{p,q}$ is independent of $h$, $\tau$ and $N$.
}
\end{theorem}

The condition of Theorem \ref{THM:BDFk} was proved in \cite[Remark 4.2]{KovacsLiLubich2016}, while the condition of Theorem \ref{THM:FEM-BDFk} is proved in the current paper; see the third result of Theorem \ref{Resolvent-Laplacian}.

\section{Main results}\label{section:main}
\setcounter{equation}{0}

Throughout this article, we assume that the coefficients $a_{ij}$ satisfy the ellipticity and regularity conditions in \eqref{Ellipticitiy}--\eqref{Regularity:aij}, and assume that the triangulation is quasi-uniform so that the Lagrange finite element spaces have all the properties in \cite[Section 3.2]{Li2019}. 

\subsection{Discrete semigroup, resolvent and maximal regularity}

For fully discrete FEMs with BDF methods in time, maximal $L^p$-regularity relies on the results in the following theorem. The proof of this theorem is presented in the next section. In \cite[Corollary 2.1]{Li2019}, the results were shown for some small $\theta\in(0,\frac{\pi}{2})$ instead of all $\theta\in(0,\frac{\pi}{2})$.  

\begin{theorem}[Estimates for the analytic semigroup and resolvent] \label{Resolvent-Laplacian}
{\it 
For all $\theta\in(0,\frac{\pi}{2})$ the following results hold:
\begin{enumerate}[label={\rm(\arabic*)},ref=\arabic*]\itemsep=5pt

\item The following collections of operators are all bounded on $L^q$ for $1\le q\le \infty$, and the bounds are independent of $h$ (but may depend on $\theta$):

(i) the semigroup $\{e^{zA_h}P_h: z\in\Sigma_{\theta}\}$ and its derivative $\{z A_h e^{zA_h}P_h: z\in\Sigma_{\theta}\};$

(ii) the resolvent operators $\{z(z-A_h)^{-1}P_h:z\in\Sigma_{\theta+\frac{\pi}{2}}\}$.

\item The semigroup of operators 
$\{e^{zA_h}P_h: z\in\Sigma_{\theta}\}$ is $R$-bounded
on $L^q$ for $1< q< \infty$, and 
the $R$-bound is independent of $h$ (but may depend on $\theta$ and $q$).

\item The collection of resolvent operators $\{z(z-A_h)^{-1}P_h:z\in\Sigma_{\theta+\frac{\pi}{2}}\}$
is $R$-bounded on $L^q$ for $1< q< \infty$, and 
the $R$-bound is independent of $h$ (but may depend on $\theta$ and $q$).  
\end{enumerate}
}
\end{theorem}

When the initial value $u_{h,0}$ is zero, by applying the Laplace transform in time to equation \eqref{FEMEq0}, one can express the solution of the semidiscrete FEM method as
\begin{align}\label{dt-uh-Fourier-expr}
\partial_t u_h = \mathcal{L}_z^{-1} [ z (z-A_h)^{-1}  P_h (\mathcal{L} f)(z)]
= \mathcal{F}_s^{-1} [ is (is -A_h)^{-1} P_h  (\mathcal{F} \tilde{f})(s)] ,
\end{align}
where $\mathcal{L}$ and $\mathcal{F}$ denote the Laplace and Fourier transforms in time, respectively, with $\tilde{f}$ denoting the zero extension of $f$ in time to $t\in\R$. The notation $\mathcal{F}_s^{-1}$ standards for the inverse Fourier transform with respect to the variable $s$. The third result of Theorem \ref{Resolvent-Laplacian} implies that the operator $M(s)=is (is -A_h)^{-1}P_h $ satisfies the $R$-boundedness on $L^q$ of the two collections of operators: 
$$
\{M(s):S\in\R\backslash\{0\} \}
\quad\mbox{and}\quad
\{s M'(s):S\in\R\backslash\{0\} \} . 
$$
This means that $M(s)$ is a Mihlin multiplier (with an $R$-bounded independent of $h$).  The Mihlin multiplier theorem (cf. \cite[Theorem 2.b]{Weis2001-1}) and the expression \eqref{dt-uh-Fourier-expr} immediately imply that 
\begin{align*}
&\|\partial_tu_h\|_{L^p(0,T;L^q)} 
\le 
C_{p,q} \|f \|_{L^p(0,T;L^q)} 
\quad \forall\,\, 1< p,q<\infty .
\end{align*} 
We summarize the result in the following corollary, which was shown in \cite[Theorem 2.1]{Li2019} by a different approach for the heat equation (with $A_h$ replaced by $\Delta_h$). 

\begin{corollary}\label{Corollary_Max_Lp_FEM}
{\it 
If $u_{h,0}=0$ and $f\in L^p(0,T;L^q)$, then the finite element solution given by \eqref{FEMEq0} has maximal $L^p$-regularity$:$ 
\begin{align}
&\|\partial_tu_h\|_{L^p(0,T;L^q)} 
+\|A_hu_h\|_{L^p(0,T;L^q)} 
\le 
C_{p,q} \|f \|_{L^p(0,T;L^q)} 
\quad \forall\,\, 1< p,q<\infty ,
\label{MaxLp1}
\end{align} 
where the constant $C_{p,q}$ is independent of $h$ and $T$.   
}
\end{corollary}

Corollary \ref{Corollary_Max_Lp_FEM} immediately implies the following error estimate between semidiscrete FEM and the PDE problem, as shown in \cite[Corolloary 2.2]{Li2019}. 

\begin{corollary}\label{Corollary_Error_FEM}
{\it 
If $f\in L^p(0,T;L^q)$, then the solutions of the semidiscrete FEM \eqref{FEMEq0} and the PDE problem \eqref{PDE1} satisfy the following estimate: 
\begin{align}
&\|u_h-P_hu\|_{L^p(0,T;L^q)}
\le
C_{p,q}(\|u-R_hu\|_{L^p(0,T;L^q)} 
+\|u_h^0-P_hu_0\|_{L^q}) 
\quad \forall\,\, 1< p,q<\infty ,
\label{Lp-Error-FEM} 
\end{align} 
where the constant $C_{p,q}$ is independent of $h$ and $T$.   
}
\end{corollary}

We present a variant version of maximal $L^p$-regularity for semidiscrete BDF methods. This variant version is often more useful than the original result in Theorem \ref{THM:BDFk} in analysis of nonlinear parabolic problems. 

\begin{theorem}\label{THM:BDF-Laplacian}
{\it 
If $u^n=0$ for $n=0,\dots,k-1$, then the semidiscrete solution given by \eqref{BDF-Laplacian-1} satisfies the following estimates, with a constant $C_{p,q}$ independent of $\tau$ and $N$. 
\begin{enumerate}
\item 
In a bounded Lipschitz domain, the maximal $L^p$-regularity estimate holds:  
\begin{equation*}
 \big\|(d_\tau u^n)_{n=k}^N\big\|_{\ell^p(L^q)} + \big\|(A  u^n )_{n=k}^N\big\|_{\ell^p(L^q)} \leq C_{p,q}\big\|(f^n)_{n=k}^N\big\|_{\ell^p(L^q)}
\hspace{42pt} \forall\, 1<p,q<\infty .
\end{equation*}

\item In a convex domain, the following additional estimate holds: 
\begin{equation*}
        \big\|(d_\tau u^n)_{n=k}^N\big\|_{\ell^p(W^{-1,q})} + \big\|(u^n )_{n=k}^N\big\|_{\ell^p(W^{1,q})} \leq C_{p,q}\big\|(f^n)_{n=k}^N\big\|_{\ell^p(W^{-1,q})}
        \quad\forall\, 1<p,q<\infty .
\end{equation*}
\end{enumerate}

}
\end{theorem}

\begin{proof}
The first result is a consequence of Theorem \ref{THM:BDFk} and \eqref{angle-condition} (which holds for all $\theta\in(0,\frac{\pi}{2})$, cf. \cite[Remark 4.2]{KovacsLiLubich2016}), together with the following equivalence relation: 
\begin{align}\label{d_tau-equivlence}
C^{-1}\|( \dot u^n )_{n=k}^N\|_{\ell^p(X)}
\le
\|(d_\tau u^n)_{n=k}^N\|_{\ell^p(X)}
\le
C\|( \dot u^n )_{n=k}^N\|_{\ell^p(X)} , 
\end{align}
where $X$ can be any UMD Banach space, including $L^q$ and $W^{-1,q}$ for $1<q<\infty$. This equivalence relation can be proved as follows. 

We consider the generating functions of $d_\tau u^n$ and $\dot u^n$, respectively, i.e.,
\begin{align*}
&\sum_{n=k}^\infty d_\tau u^n \zeta^n  = \frac{1-\zeta}{\tau} \sum_{n=k}^\infty u^n\zeta^n  ,\\
&\sum_{n=k}^\infty \dot u^n \zeta^n
=
\frac{1}{\tau} \delta(\zeta) \sum_{n=k}^\infty u^n\zeta^n ,
\end{align*}
where $\zeta$ is on the unit disk of the complex plane. 
The two equalities above imply 
\begin{align*}
& \sum_{n=k}^\infty d_\tau u^n \zeta^n  = \frac{1-\zeta}{\delta(\zeta)} \sum_{n=k}^\infty \dot u^n \zeta^n .
\end{align*}
By substituting $\zeta=e^{-i\theta}$ into the equality above, we obtain 
\begin{align*}
&[( d_\tau u^n)_{n=k}^\infty ]
= \mathcal{F}_\theta^{-1} \frac{1-e^{i\theta}}{\delta(e^{i\theta})} \mathcal{F}  [( \dot u^n)_{n=k}^\infty ](\theta) , \\
&[( \dot u^n)_{n=k}^\infty ]
= \mathcal{F}_\theta^{-1} \frac{\delta(e^{i\theta})}{1-e^{i\theta}} \mathcal{F}  [( d_\tau u^n)_{n=k}^\infty ](\theta) ,
\end{align*}
where $\mathcal{F}$ represents the Fourier transform (which transforms a sequence to a Fourier series) and $\mathcal{F}_\theta^{-1}$ is its inverse transform (which transforms a function of $\theta$ to a sequence). 

Since the polynomial $\delta(\zeta)$ associated to the $k$-step BDF method has only one zero at $\zeta=1$ on the closed unit disk of the complex plane, it follows that the functions 
$$
M_1(\zeta)=\frac{1-\zeta}{\delta(\zeta)}
\quad\mbox{and}\quad
M_2(\zeta)=\frac{\delta(\zeta)}{1-\zeta}
$$ 
are bounded from both below and above on the unit circle of the complex plane, therefore both satisfying Blunck's multiplier conditions (cf. \cite[Theorem 4]{JinLiZhou-MaxLp} or \cite[Theorem 1.3]{Blunck2001}), i.e.,
\begin{align*}
|M_j(\zeta)|\le C,\quad\mbox{and}\quad
|(1+\zeta)(1-\zeta)M_j'(\zeta)| \le C ,\quad j=1,2,
\quad\forall\, \zeta\in\C\,\,\,\mbox{such that}\,\,\,|\zeta|=1,\,\,\zeta\neq 1.
\end{align*}
Hence, the operators $\mathcal{F}_\theta^{-1} \frac{1-e^{i\theta}}{\delta(e^{i\theta})} \mathcal{F} $ and $\mathcal{F}_\theta^{-1} \frac{\delta(e^{i\theta})}{1-e^{i\theta}} \mathcal{F} $ are both bounded on $\ell^p(X)$. As a result, we have 
\begin{align*}
&\| ( d_\tau u^n)_{n=k}^\infty \|_{\ell^p(X)}
\le
C \| ( \dot u^n)_{n=k}^\infty \|_{\ell^p(X)}
\quad\mbox{and}\quad 
\| ( \dot u^n)_{n=k}^\infty \|_{\ell^p(X)}
\le
C \| ( d_\tau u^n)_{n=k}^\infty \|_{\ell^p(X)} . 
\end{align*}
By modifying $u^n$ for $n\ge N+1$, we can make $\dot u^n=0$ for $n\ge N+1$ without changing the values of $d_\tau u^n$ for $k\le n\le N$. Then the first of the two inequalities above implies
\begin{align*}
&\| ( d_\tau u^n)_{n=k}^N \|_{\ell^p(X)}
\le
C \| ( \dot u^n)_{n=k}^N \|_{\ell^p(X)}
\end{align*}
Similarly, by modifying $u^n$ for $n\ge N+1$, we can make $d_\tau u^n=0$ for $n\ge N+1$ without changing the values of $\dot u^n$ for $k\le n\le N$. Then we obtain  
\begin{align*}
\| ( \dot u^n)_{n=k}^N \|_{\ell^p(X)}
\le
C \| ( d_\tau u^n)_{n=k}^N \|_{\ell^p(X)} . 
\end{align*}
This proves the norm equivalence \eqref{d_tau-equivlence} and completes the proof of the first result. 

The second result was proved in \cite[Proposition 8.2]{AkrivisLiLubich2017} for $q_d'<q<q_d$, where $q_d$ is the maximal number such that the solution of the Poisson equation
\begin{align}\label{Poisson-Eq}
\left\{
\begin{aligned}
&A v=g &&\mbox{in}\,\,\,\varOmega,\\
&v=0 &&\mbox{on}\,\,\,\partial\varOmega,
\end{aligned}
\right.
\end{align}
satisfies estimate
\begin{align}\label{W1q-Poisson}
\|v\|_{W^{1,q}}
\le
C\|g\|_{W^{-1,q}}
\quad\forall\, q_d'<q<q_d .
\end{align}
In a convex domain we have $q_d=\infty$ (cf. 
\cite[Theorem 1]{Jia-Li-Wang-2010}). This implies the second result.  
\end{proof}
\medskip

Theorem \ref{THM:FEM-BDFk} and the third result of Theorem \ref{Resolvent-Laplacian}, together with the equivalence relation \eqref{d_tau-equivlence}, imply the following result for the fully discrete FEM. 

\begin{theorem}\label{THM:maximal-Lp-Laplacian}
{\it 
If $u_h^n=0$ for $n=0,\dots,k-1$, then the fully discrete solution given by \eqref{BDF-FEM} satisfies the discrete maximal $L^p$-regularity estimate:  
\begin{equation}\label{Maximal-Lp-FEM}
        \big\|(d_\tau u_h^n)_{n=k}^N\big\|_{\ell^p(L^q)} + \big\|(A_h  u_h^n )_{n=k}^N\big\|_{\ell^p(L^q)} \leq C_{p,q}\big\|(f_h^n)_{n=k}^N\big\|_{\ell^p(L^q)}
        \quad\forall\, 1<p,q<\infty ,
\end{equation}
where the constant $C_{p,q}$ is independent of $h$, $\tau$ and $N$.
}
\end{theorem}

\subsection{Error estimate between semi- and fully discrete solutions}

Error estimates of semidiscrete BDF methods by using maximal $L^p$-regularity were established for semilinear and quasi-linear parabolic equations in \cite{AkrivisLi2018} and \cite{AkrivisLiLubich2017,KunstmanLiLubich2018}, respectively. In this subsection, we present a tool for establishing error estimates between fully discrete and semidiscrete solutions. 

\begin{theorem}\label{THM:LpLq-Laplacian} 
{\it 
If $f_h^n=P_hf^n$, then the fully discrete solution given by \eqref{BDF-FEM} and the semidiscrete solution given by \eqref{BDF-Laplacian-1} satisfy the following error estimate: 
\begin{align}\label{LpLq-Laplacian} 
        \big\|(P_hu^n-u_h^n )_{n=k}^N\big\|_{\ell^p(L^q)} 
        \le C_{p,q}\big\|(P_h u^n - R_hu^n )_{n=k}^N\big\|_{\ell^p(L^q)} 
      +C_{p,q} \sum_{n=0}^{k-1}\big\|P_h u^n - u_h^n \big\|_{L^q} , \\ 
\forall\, 1<p,q<\infty , \notag 
\end{align} 
where $R_h:H^1_0\rightarrow S_h$ is the Ritz projection defined by 
$$ 
\sum_{i,j=1}^d (a_{ij}\partial_j(w-R_hw),\partial_iv_h)=0\quad\forall\, v_h\in S_h,\,\,\,w\in H^1_0 ,
$$ 
and the constant $C_{p,q}$ is independent of $h$, $\tau$ and $N$. 
} 
\end{theorem} 

\begin{proof} 
It is known that the Ritz projection defined above satisfies the following identity: 
$$P_h A=A_hR_h \quad\mbox{on}\,\,\,H_0^1.$$ By applying the operator $P_h$ to \eqref{BDF-Laplacian-1} and using the identity above, we obtain 
\begin{align}\label{BDF-Laplacian-2}
\frac{1}{\tau} \sum_{j=0}^k \delta_j P_hu^{n-j}
=A_hR_hu^n+P_hf^n, \quad n \geq k .
\end{align}
Then, subtracting \eqref{BDF-FEM} from \eqref{BDF-Laplacian-2}, we have
\begin{align}\label{BDF-FEM-Lapalcian-Error}
\frac{1}{\tau} \sum_{j=0}^k \delta_j (P_hu^{n-j}-u_h^{n-j})
- A_h(P_hu^n-u_h^n) = A_h (R_hu^n-P_hu^n), \quad n \geq k .
\end{align}
Applying the operator $A_h^{-1}$ to the equation above, we obtain
\begin{align}\label{BDF-FEM-Lapalcian-Error}
\frac{1}{\tau} \sum_{j=0}^k \delta_j [A_h^{-1}(P_hu^{n-j}-u_h^{n-j})]
- A_h [A_h^{-1}(P_hu^n-u_h^n) ]=  R_hu^n-P_hu^n, \quad n \geq k .
\end{align}

We decompose the solution of \eqref{BDF-FEM-Lapalcian-Error} into two parts, i.e., 
\begin{align}\label{d-wh-vh}
A_h^{-1}(P_hu^{n}-u_h^{n})=w_h^n+v_h^n , 
\end{align}
with 
\begin{align}\label{BDF-FEM-whn-j-Error1}
\left\{
\begin{aligned}
&\frac{1}{\tau} \sum_{j=0}^k \delta_j w_h^{n-j}
- A_h w_h^n =  R_hu^n-P_hu^n, && n \geq k ,\\
&w_h^n=0, &&n=0,\dots,k-1,
\end{aligned}
\right.
\end{align}
and
\begin{align}\label{BDF-FEM-Lapalcian-Error1}
\left\{
\begin{aligned}
&\frac{1}{\tau} \sum_{j=0}^k \delta_j v_h^{n-j}
- A_h v_h^n =  0, && n \geq k ,\\
&v_h^n=A_h^{-1}(P_hu^{n}-u_h^{n}),  &&n=0,\dots,k-1 .
\end{aligned}
\right.
\end{align}
Then Theorem \ref{THM:maximal-Lp-Laplacian} implies the following estimate for $w_h^n$: 
\begin{equation}\label{Estimate-wh}
\big\|(A_hw_h^n )_{n=k}^N\big\|_{\ell^p(L^q)} \leq C_{p,q}\big\|(R_hu^n-P_hu^n)_{n=k}^N\big\|_{\ell^p(L^q)}
        \quad\forall\, 1<p,q<\infty .
\end{equation}
The solution of \eqref{BDF-FEM-Lapalcian-Error1} can be represented by the starting values as (cf. \cite[Lemma 10.3]{Thomee2006} or Lemma \ref{Lemma:Eh} below) 
\begin{equation}\label{Repr-Forml}
v_h^n
= 
\sum_{j=0}^{k-1} E_h^{n,j} A_h^{-1}(P_hu^j-u_h^j) ,\quad\text{for}\,\,\, n\geqslant k ,
\end{equation}
where $E_h^{n,j}:S_h\rightarrow S_h$ are some operators satisfying the following estimates (for some constant $\lambda_0>0$; see Lemma \ref{Lemma:Eh}): 
\begin{align}\label{beta-bound}
\begin{aligned}
&\|E_h^{n,j}\phi_h\|_{L^q} \le Ce^{-\lambda_0 t_{n+1} } \|\phi_h\|_{L^q} &&\forall\, \phi_h\in S_h, \,\,\,\forall\, 1\le q\le \infty, \\
&\|A_hE_h^{n,j}\phi_h\|_{L^q} \le Ce^{-\lambda_0 t_{n+1} }\|A_h\phi_h\|_{L^q} &&\forall\, \phi_h\in S_h, \,\,\,\forall\, 1\le q\le \infty, \\
&\|A_hE_h^{n,j}\phi_h\|_{L^q} \le Ce^{-\lambda_0 t_{n+1} } t_{n+1}^{-1}\|\phi_h\|_{L^q}  
&&\forall\, \phi_h\in S_h , \,\,\,\forall\, 1\le q\le \infty . 
\end{aligned}
\end{align}
Substituting the second and third estimates of \eqref{beta-bound} into \eqref{Repr-Forml}, we obtain 
\begin{align*}
\begin{aligned}
&\| (A_hv_h)_{n=k}^N \|_{\ell^\infty(L^q)} \le C\sum_{j=0}^{k-1}\| P_hu^j-u_h^j \|_{L^q} , \\
&\| (A_hv_h)_{n=k}^N \|_{\ell^{1,\infty}(L^q)} 
\le C\sum_{j=0}^{k-1} \|A_h^{-1}(P_hu^j-u_h^j)\|_{L^q} \le C\sum_{j=0}^{k-1}\| P_hu^j-u_h^j \|_{L^q} ,
\end{aligned}
\end{align*}
where $\ell^{1,\infty}$ is the weak-type $\ell^1$; see \cite[Section 1.1]{Grafakos2008}. 
By the real interpolation between $\ell^{1,\infty}(L^q)$ and $\ell^{\infty}(L^q)$ (cf. \cite[Proposition 1.1.14]{Grafakos2008}), we have 
\begin{align}\label{Estimate-vh}
\| (A_hv_h)_{n=k}^N \|_{\ell^p(L^q)}
&\le
C\| (A_hv_h)_{n=k}^N \|_{\ell^{1,\infty}(L^q)}^{\frac1p}
\| (A_hv_h)_{n=k}^N \|_{\ell^\infty(L^q)}^{1-\frac1p} \notag \\ 
&\le
C\sum_{j=0}^{k-1}\| P_hu^j-u_h^j \|_{L^q} .
\end{align}
Then, substituting \eqref{Estimate-wh} and \eqref{Estimate-vh} into \eqref{d-wh-vh}, we obtain the desired estimate \eqref{LpLq-Laplacian}. 

In this proof we have used \eqref{beta-bound}, which is proved in the following lemma.
\hfill
\end{proof}

\begin{lemma}[Estimates of the solution operator for fully discrete BDF methods]\label{Lemma:Eh}
{\it
The solution of \eqref{BDF-FEM-Lapalcian-Error1} can be represented by \eqref{Repr-Forml}, where the operator $E_h^{n,j}$ satisfies the estimates in \eqref{beta-bound}, with a constant $C$ independent of $h$, $\tau$, $n\ge k$ and $1\le q\le \infty$. 
}
\end{lemma}
\begin{proof}
By setting $v_h^n=0$ for $n=0,\dots,k-1$ without changing the values of $v_h^n$ for $n\ge k$, equation \eqref{BDF-FEM-Lapalcian-Error1} can be rewritten as 
\begin{align}\label{BDF-FEM-Lapalcian-Error2}
\left\{
\begin{aligned}
&\frac{1}{\tau} \sum_{j=0}^k \delta_j v_h^{n-j}
- A_h v_h^n = g_h^n , && n \geq k ,\\[5pt]
&v_h^n=0 &&n=0,\dots,k-1,
\end{aligned}
\right.
\end{align}
with 
$$
g_h^n=
\left\{\begin{aligned}
&-\frac{1}{\tau} \sum_{j=n-k}^{k-1} \delta_{n-j} A_h^{-1} (P_hu^{j}-u_h^{j})
&&\mbox{for}\,\,\,
k\le n\le 2k-1 , \\
&0&&\mbox{for}\,\,\, n\ge 2k .
\end{aligned}\right.
$$ 
Without loss of generality, we can also set $v_h^n=0$ for $n\ge N+1$ without affecting equation \eqref{BDF-FEM-Lapalcian-Error2} and $g_h^n$ for $k\le n\le N$. In the following, we derive expression of $v_h^n$ for $k\le n\le N$. But the expression turns out to be independent of $N$ and therefore holds for all $n\ge k$. 

We denote by 
$$v(\zeta)=\sum_{n=k}^\infty v_h^n\zeta^n\quad \forall\,\zeta\in\C ,$$ 
the generating function of the sequence $(v_h^n)_{n=k}^\infty$. The series is well defined as an analytic function on the complex plane, because there are only a finite number of $v_h^n$ that are not zero (after modifying $v_h^n$ to be zero for $n\ge N+1$). 
Multiplying \eqref{BDF-FEM-Lapalcian-Error2} by $\zeta^n$ and summing up the results for $n=k,k+1,\dots$, we obtain
\begin{align*}
(\tau^{-1}\delta(\zeta)-A_h) v(\zeta) =   \sum_{m=k}^{2k-1}g_h^m\zeta^m ,
\end{align*}
which implies
\begin{align}
v(\zeta)  = (\tau^{-1}\delta(\zeta)-A_h)^{-1}   \sum_{m=k}^{2k-1}g_h^m\zeta^m .
\end{align}

It is known that, by choosing $\kappa\in(\frac{\pi}{2},\pi)$ sufficiently close to $\frac{\pi}{2}$ (independent of $\tau$, cf. \cite[Lemma B.1]{Jin-Li-Zhou-2017-BDF}), the polynomial $\delta(\zeta)$ associated to the $k$-step BDF method satisfies 
\begin{align}\label{angle-delta-z}
\begin{aligned}
&\delta(e^{z\tau})\in \varSigma_{\pi-\frac12\alpha_k} 
&&\forall\, z\in \varSigma_{\kappa}^\tau =\{z\in \varSigma_{\kappa}: |{\rm Im}(z)|\le \pi/\tau\} ,  \\
&C^{-1}|z|\leqslant | \tau^{-1} \delta(e^{z\tau})|\leqslant C|z| &&\forall\, z\in -\lambda_0+\varSigma_{\kappa}^\tau,\\
&|z-\tau^{-1}\delta(e^{z\tau})|\leqslant C|z|^2\tau
&&\forall\, z\in -\lambda_0+\varSigma_{\kappa}^\tau , 
\end{aligned}
\end{align}
for some sufficiently small constant $\lambda_0$; see the footnote.\blfootnote{In \cite[Lemma B.1]{Jin-Li-Zhou-2017-BDF} these estimates were all proved only for $z\in \varSigma_{\kappa_*}^\tau$  instead of $z\in -\lambda_0+\varSigma_{\kappa_*}^\tau$, for all angles $\kappa_*$ sufficiently close to $\frac{\pi}{2}$. Nevertheless, by using Taylor's expansion one can see that these estimates are still correct when $|z|$ is sufficiently small (smaller than some constant independent of $\tau$). Hence, these estimates also holds for $z\in -\lambda_0+\varSigma_{\kappa}^\tau$ with a sufficiently small constant $\lambda_0$ and an angle $\kappa$ slightly closer to $\frac{\pi}{2}$ than $\kappa_*$.} 
Theorem \ref{Resolvent-Laplacian} and \eqref{angle-delta-z}  imply that $(\tau^{-1}\delta(\zeta)-A_h)^{-1}$ is analytic and satisfies the following estimate for $z\in \varSigma_{\kappa}^\tau$:
\begin{align*}
&\|(\tau^{-1}\delta(e^{z\tau})-A_h)^{-1}\|_{L^q\rightarrow L^q}
\le C|z|^{-1} , 
\quad
\|A_h(\tau^{-1}\delta(e^{z\tau})-A_h)^{-1}\|_{L^q\rightarrow L^q}
\le C .
\end{align*}
Since the largest eigenvalue of $A_h$ is strictly negative and bounded away from zero (with an upper bound independent of $h$), it follows that in a small neighborhood of $z=0$ the operator $(z-A_h)^{-1}$ is bounded analytic. As a result, the resolvent estimates above can be slightly improved as follows (for sufficiently small constant $\lambda_0$):
\begin{align}\label{resolvent-delta}
\begin{aligned}
&\|(\tau^{-1}\delta(e^{z\tau})-A_h)^{-1}\|_{L^q\rightarrow L^q}
\le 
C(1+|z|)^{-1} \le C|z+\lambda_0|^{-1} &&\forall\,z\in-\lambda_0 + \varSigma_{\kappa}^\tau,  \\[3pt]
&
\|A_h(\tau^{-1}\delta(e^{z\tau})-A_h)^{-1}\|_{L^q\rightarrow L^q}
\le C &&\forall\,z\in-\lambda_0 + \varSigma_{\kappa}^\tau .
\end{aligned}
\end{align}

By the Cauchy integral formula, we have 
\begin{align}\label{expr-vh-1}
v_h^n  
&=  \frac{1}{2\pi i}\int_{|\zeta|=1} 
(\tau^{-1}\delta(\zeta)-A_h)^{-1} \sum_{m=k}^{2k-1}  g_h^m\zeta^m \frac{\d\zeta}{\zeta^{n+1}} \notag \\
&= \frac{\tau}{2\pi i}\int_{|{\rm Im}(z)|\le \frac{\pi}{\tau}} 
(\tau^{-1}\delta(e^{-\tau z})-A_h)^{-1}  \sum_{m=k}^{2k-1}  g_h^me^{-t_m z} e^{t_nz}\d z ,
&&\mbox{(change of variable $\zeta=e^{-\tau z}$)} .
\notag \\
&= \sum_{m=k}^{2k-1} \frac{\tau}{2\pi i}\int_{\varGamma^\tau_{\kappa,n}} 
(\tau^{-1}\delta(e^{-\tau z})-A_h)^{-1}    g_h^m e^{t_{n-m} z}\d z  && \mbox{for}\,\,\, n\ge k ,
\end{align}
where we have deformed the integration contour to 
\begin{align}\label{contour-Gamma}
&\varGamma^\tau_{\kappa,n} = \varGamma^{\tau,1}_{\kappa,n} \cup \varGamma^{\tau,2}_{\kappa,n}  \\
&\mbox{with}\,\,\,
\varGamma^{\tau,1}_{\kappa,n}=\{z\in\C: \arg(z+\lambda_0) = \pm \kappa ,\,\,\,|z+\lambda_0|\ge t_{n+1}^{-1} , \,\,\, |{\rm Im}(z)| \le \pi/\tau \} \notag \\
\quad
&\mbox{and}\,\,\, 
\varGamma^{\tau,2}_{\kappa,n}=\{z\in\C: |\arg(z+\lambda_0)| \le \kappa ,\,\,\,|z+\lambda_0| = t_{n+1}^{-1} \} . \notag 
\end{align}
The deformation of the integration contour in \eqref{operator:M} is legal due to the analyticity of the integrand for $z\in -\lambda_0+\Sigma_\kappa^\tau$ and the periodicity in the imaginary part of $z$. We define the operator 
\begin{align}\label{operator:M}
M_{n-m}
&
=\frac{1}{2\pi i}\int_{\Gamma_{\kappa,k+n-m}^\tau} 
(\delta(e^{-\tau z})-A_h)^{-1} e^{t_{n-m}z} \d z \notag \\
&
=\frac{1}{2\pi i}\int_{\Gamma_{\kappa,n}^\tau} 
(\delta(e^{-\tau z})-A_h)^{-1} e^{t_{n-m}z} \d z .
&& \mbox{(contour is deformed)}
\end{align}
Then from \eqref{expr-vh-1} we obtain
\begin{align*}
v_h^n  
&=\sum_{m=k}^{2k-1}  \sum_{j=m-k}^{k-1}M_{n-m}\delta_{m-j} A_h^{-1} (P_hu^{j}-u_h^{j}) \\
&= \sum_{j=0}^{k-1} \sum_{m=k}^{j+k}  M_{n-m}\delta_{m-j}  A_h^{-1} (P_hu^{j}-u_h^{j}) \qquad\mbox{for}\,\,\, n\ge k .
\end{align*}
Therefore, the operator $E_h^{n,j}$ in \eqref{Repr-Forml} is given by 
\begin{align}\label{expr-Eh-nj}
E_h^{n,j}
=\sum_{m=k}^{j+k}  M_{n-m}\delta_{m-j} . 
\end{align}
By using \eqref{resolvent-delta} and the property $\tau |z|\le C$ on $\varGamma^\tau_{\kappa,n}$ with $k\le m\le 2k-1$, from \eqref{operator:M} we derive that 
\begin{align}\label{contour-estimate-1}
&e^{\lambda_0t_{n+1}} \|M_{n-m} \phi_h\|_{L^q} \notag \\
&\le 
C \sup_{z\in\varGamma^\tau_{\kappa,n}} e^{- (m+1) \tau {\rm Re}(z)}  \int_{\varGamma^\tau_{\kappa,n}} 
\| ( \delta(e^{-\tau z})-A_h)^{-1} \|_{L^q\rightarrow L^q} 
e^{t_{n+1}{\rm Re}(z+\lambda_0)}  \|\phi_h\|_{L^q} |\d z| \notag \\
&\le 
C\|\phi_h\|_{L^q}  
\bigg(\int_{\varGamma^{\tau,1}_{\kappa,n}}  
|z+\lambda_0|^{-1}  
e^{-Ct_{n+1} |z+\lambda_0|} |\d (z+\lambda_0) | 
+\int_{\varGamma^{\tau,2}_{\kappa,n}}  
|z+\lambda_0|^{-1}  
e^{t_{n+1} |z+\lambda_0|} |\d (z+\lambda_0) |\bigg) \notag \\
&\le 
C\|\phi_h\|_{L^q}  
\bigg( \int_{t_{n+1}^{-1}}^{\frac{\pi }{\tau \sin(\theta)} } 
r^{-1}  
e^{-Ct_{n+1} r}\d r + \int_{-\kappa}^{\kappa}
t_{n+1} 
e^C t_{n+1}^{-1} \d\varphi \bigg) \notag \\
&\le C\|\phi_h\|_{L^q}  \qquad\,\mbox{for}\,\,\, k\le m\le 2k-1\,\,\,\mbox{and}\,\,\, n\ge k .
\end{align}
Similarly, we have 
\begin{align}
&e^{\lambda_0t_{n+1}} \|A_hM_{n-m} \phi_h\|_{L^q} \notag \\
&\le 
C \sup_{z\in\varGamma^\tau_{\kappa,n}} e^{- (m+1) \tau {\rm Re}(z)}  \int_{\varGamma^\tau_{\kappa,n}} 
\| ( \delta(e^{-\tau z})-A_h)^{-1} \|_{L^q\rightarrow L^q} 
e^{t_{n+1}{\rm Re}(z+\lambda_0)} \|A_h\phi_h\|_{L^q}  \d z \notag \\
&\le
C\|A_h\phi_h\|_{L^q}  \bigg(\int_{\varGamma^{\tau,1}_{\kappa,n}} 
 |z+\lambda_0|^{-1}  
e^{-C t_{n+1} |z+\lambda_0|} |\d (z+\lambda_0) | 
+\int_{\varGamma^{\tau,2}_{\kappa,n}} 
 |z+\lambda_0|^{-1}  
e^{t_{n+1} |z+\lambda_0|} |\d (z+\lambda_0) | \bigg)
\notag \\
&\le C\|A_h\phi_h\|_{L^q} \quad\mbox{for}\,\,\, k\le m\le 2k-1\,\,\,\mbox{and}\,\,\, n\ge k,
\end{align}
and 
\begin{align}\label{contour-estimate-3}
&e^{\lambda_0t_{n+1}}\|A_hM_{n-m} \phi_h\|_{L^q} \notag \\
&\le 
C \sup_{z\in\varGamma^\tau_{\kappa,n}} e^{- (m+1) \tau {\rm Re}(z)}  \int_{\varGamma^\tau_{\kappa,n}} 
\| A_h ( \delta(e^{-\tau z})-A_h)^{-1} \|_{L^q\rightarrow L^q} 
e^{t_{n+1}{\rm Re}(z+\lambda_0)}  \|\phi_h\|_{L^q} \d z \notag \\
&\le
C\|\phi_h\|_{L^q} \bigg(\int_{\varGamma^{\tau,1}_{\kappa,n}} 
e^{-C t_{n+1} |(z+\lambda_0)|} |\d (z+\lambda_0) | 
+\int_{\varGamma^{\tau,2}_{\kappa,n}} 
e^{t_{n+1} |(z+\lambda_0)|} |\d (z+\lambda_0) |\bigg)\notag \\
&\le
C \|\phi_h\|_{L^q} 
\bigg( \int_{t_{n+1}^{-1}}^\infty e^{-Ct_{n+1} r} \d r
+\int_{-\kappa}^{\kappa} e^{C} t_{n+1}^{-1} \d \varphi \bigg)  \notag \\
&\le C t_{n+1}^{-1} \|\phi_h\|_{L^q} \quad\mbox{for}\,\,\, k\le m\le 2k-1\,\,\,\mbox{and}\,\,\, n\ge k.
\end{align}
Substituting \eqref{contour-estimate-1}--\eqref{contour-estimate-3} into \eqref{expr-Eh-nj}, we obtain the desired estimates in \eqref{beta-bound}. 
\end{proof}
\medskip

\begin{remark}
The shift of the region from $\Sigma_\kappa^\tau$ to $-\lambda_0+\Sigma_\kappa^\tau$ is to have the exponential factor $e^{-\lambda_0 t_{n+1} }$ in the estimates of \eqref{beta-bound}. This exponential factor plays a role in establishing the $\ell^\infty(L^q)$ error estimate in the following theorem. 
\end{remark}

\begin{theorem}\label{THM:Linfty-stability}
{\it 
If $f_h^n=P_hf^n$, then the fully discrete solution given by \eqref{BDF-FEM} and the semidiscrete solution given by \eqref{BDF-Laplacian-1} satisfy the following error estimate for $1\le q\le \infty$: 
\begin{align}\label{Linfty-stability}
\max_{k\le n\le N} \|P_hu^n-u_h^n\|_{L^q} 
\le C\ell_{N} \max_{k\le n\le N} \|P_h u^n - R_hu^n\|_{L^q} 
+ C\max_{0\le n\le k-1} \big\|P_h u^n - u_h^n \big\|_{L^q} ,
\end{align}
where $\ell_{N}=  \ln (1+N) $ and the constant $C$ is independent of $h$, $\tau$ and $N$.
}
\end{theorem}
\begin{proof}
Again, we use the error equation \eqref{BDF-FEM-Lapalcian-Error} and the decomposition \eqref{d-wh-vh}. Then, substituting the second estimate of \eqref{beta-bound} into the expression \eqref{Repr-Forml}, we obtain 
\begin{align}\label{Ah-vh-n-Linfty}
&\sup_{k\le n\le N } \| A_hv_h^n \|_{L^q} 
\le C\max_{0\le n\le k-1} \| P_hu^n-u_h^n \|_{L^q} .
\end{align}
Similarly, we can express the solution of \eqref{BDF-FEM-whn-j-Error1} as 
\begin{align}\label{BDF-expr-w}
w_h^n = \tau \sum_{j=k}^n E_h^{n-j} ( R_hu^j-P_hu^j ) ,
\end{align}
with some operators $E_h^n$ satisfying the following estimate for $1\le q\le \infty$ (for some positive constant $\lambda_0$; see Lemma \ref{Lemma:Ehn-Linfty}):
\begin{align}\label{Estimate-Ehn}
\begin{aligned}
&\|E_h^n \phi_h\|_{L^q}
\le
Ce^{-\lambda_0 t_{n+1}} \|\phi_h\|_{L^q}  &&
\forall\,
\phi_h\in S_h, \\
&\|A_h E_h^n \phi_h\|_{L^q}
\le
Ce^{-\lambda_0 t_{n+1}} t_{n+1}^{-1} \|\phi_h\|_{L^q} 
&&
\forall\,
\phi_h\in S_h. 
\end{aligned}
\end{align} 
Substituting the second estimate of \eqref{Estimate-Ehn} into \eqref{BDF-expr-w}, we obtain 
\begin{align}\label{log-tau}
\|A_hw_h^n\|_{L^q}
&\le
\tau \sum_{j=k}^n Ce^{-\lambda_0 t_{n+1-j}} t_{n+1-j}^{-1} \max_{k\le j\le n}\| R_hu^j - P_hu^j \|_{L^q} \notag \\
&\le
C \ln (1+n)  \max_{k\le j\le n} \| R_hu^j - P_hu^j \|_{L^q} .
\end{align} 

It remains to prove the estimates in \eqref{Estimate-Ehn}. This is presented in the following lemma. 
\end{proof}

\begin{lemma}\label{Lemma:Ehn-Linfty}
{\it 
The solution of \eqref{BDF-FEM-whn-j-Error1} can be expressed as \eqref{BDF-expr-w}, where the operators $E_h^n$ satisfy the estimates in \eqref{Estimate-Ehn}  for $1\le q\le \infty$. 
}
\end{lemma}
\begin{proof}
Similarly as the proof of Lemma \ref{Lemma:Eh}, we multiply \eqref{BDF-FEM-whn-j-Error1} by $\zeta^n$ and sum up the results for all $n\ge k$. This yields the following expression of the generating function $w(\zeta)=\sum_{n=k}^\infty w_h^n\zeta^n$: 
\begin{align}
w(\zeta) = (\tau^{-1}\delta(\zeta)-A_h)^{-1}   \sum_{n=k}^{\infty} (R_hu^n - P_hu^n) \zeta^n .
\end{align}
Then, by using Cauchy's integral formula, we can derive the following expression similarly as \eqref{expr-vh-1}:
\begin{align}\label{expr-wh-1}
w_h^n  
&=\sum_{m=k}^{n} \frac{\tau}{2\pi i}\int_{\varGamma^\tau_{\kappa,{n-m}}} 
(\tau^{-1}\delta(e^{-\tau z})-A_h)^{-1}   (R_hu^m - P_hu^m) e^{t_{n-m} z}\d z  && \mbox{for}\,\,\, n\ge k ,
\end{align}
where the integration contour $\varGamma^\tau_{\kappa,n}$ is defined in \eqref{contour-Gamma}, and the summation is from $m=k$ to $m=n$ because for $m\ge n+1$ the integral above become zero (the solution $w_h^n$ only depends on the right-hand side for $k\le m\le n$). Hence, the operator in \eqref{BDF-expr-w} is given by
\begin{align}\label{expr-wh-1}
E_h^{n}   
&= \frac{1}{2\pi i}\int_{\varGamma^\tau_{\kappa,n}} 
(\tau^{-1}\delta(e^{-\tau z})-A_h)^{-1} e^{t_{n} z} \d z  && \mbox{for}\,\,\, n\ge k .
\end{align}
By using the resolvent estimates in \eqref{resolvent-delta} and the similar estimation in \eqref{contour-estimate-1}--\eqref{contour-estimate-3}, we have 
\begin{align*}
&e^{\lambda_0t_{n+1}} \| E_h^{n} \phi_h \|_{L^q} \\
&\le 
C \sup_{z\in\varGamma_{\kappa,n}^\tau} e^{- \tau {\rm Re}(z)}  \int_{\varGamma_{\kappa,n}^\tau} 
\| ( \delta(e^{-\tau z})-A_h)^{-1} \|_{L^q\rightarrow L^q} 
e^{t_{n+1}{\rm Re}(z+\lambda_0)}  \|\phi_h\|_{L^q} \d z \\
&\le
C\|\phi_h\|_{L^q}  
\bigg(\int_{\varGamma_{\kappa,n}^{\tau,1}} 
 |z+\lambda_0|^{-1}  
e^{-Ct_{n+1} |z+\lambda_0|} |\d (z+\lambda_0) |
+\int_{\varGamma_{\kappa,n}^{\tau,2}} 
 |z+\lambda_0|^{-1}  
e^{t_{n+1} |z+\lambda_0|} |\d (z+\lambda_0) |
\bigg) \\
&\le 
C\|\phi_h\|_{L^q}  
\bigg( \int_{t_{n+1}^{-1}}^{\infty} 
r^{-1}  
e^{-Ct_{n+1} r}\d r + \int_{-\kappa}^\kappa  
e^{C} \d\varphi \bigg) \notag \\
&\le C\|\phi_h\|_{L^q}  \qquad\,\mbox{for}\,\,\, n\ge k,
\end{align*}
and
\begin{align*}
e^{\lambda_0t_{n+1}}\|A_hE_h^{n} \phi_h\|_{L^q}
&\le 
C \sup_{z\in \varGamma_{\kappa,n}^\tau} e^{- \tau {\rm Re}(z)}  \int_{\varGamma_{\kappa,n}^\tau} 
\|A_h  ( \delta(e^{-\tau z})-A_h)^{-1} \|_{L^q\rightarrow L^q} 
e^{t_{n+1}{\rm Re}(z+\lambda_0)} \|\phi_h\|_{L^q}  \d z \\
&\le
C\|A_h\phi_h\|_{L^q}  
\bigg( \int_{\varGamma_{\kappa,n}^{\tau,1}} 
e^{-Ct_{n+1} |z+\lambda_0|} |\d (z+\lambda_0) |
+\int_{\varGamma_{\kappa,n}^{\tau,2}} 
e^{t_{n+1} |z+\lambda_0|} |\d (z+\lambda_0) |
\bigg) \\
&\le 
C\|\phi_h\|_{L^q}  
\bigg( \int_{t_{n+1}^{-1}}^{\infty} 
r^{-1}  
e^{-Ct_{n+1} r}\d r + \int_{-\kappa}^\kappa  
t_{n+1}^{-1} e^{C} \d\varphi \bigg) \notag \\
&\le Ct_{n+1}^{-1} \|\phi_h\|_{L^q} \quad\mbox{for}\,\,\, n\ge k .
\end{align*}
This proves the desired estimates in \eqref{Estimate-Ehn}. 
\end{proof}

\subsection{$\ell^p(W^{1,q})$ estimate for fully discrete FEMs}

In this subsection, we prove the following result by using Theorems \ref{THM:BDF-Laplacian} and \ref{THM:LpLq-Laplacian}. 

\begin{theorem}\label{THM:LpW1q-Laplacian}
{\it 
In a convex polygon or polyhedron, the solution of \eqref{BDF-FEM} with zero starting values $u_h^n=0$, $n=0,\dots,k-1$, satisfies the following estimate: 
    \begin{equation}\label{LpW1q-Laplacian}
        \big\|(d_\tau u_h^n )_{n=k}^N\big\|_{\ell^p(W^{-1,q})}   
        +\big\|(u_h^n )_{n=k}^N\big\|_{\ell^p(W^{1,q})}   
        \le C_{p,q}\big\|(f_h^n )_{n=k}^N\big\|_{\ell^p(W^{-1,q})}  ,
    \end{equation}
where the constant $C_{p,q}$ is independent of $h$ and $\tau$.
}
\end{theorem}
\begin{proof}
Part 1: 
For $2\le q<\infty$, the Ritz projection has the following approximation property 
in a convex polygon or polyhedron (see Remark \ref{Remark:Lq-Ritz}):
\begin{align}\label{Ritz-Laplacian-Lq}
\|P_hu^n-R_hu^n\|_{L^q}\le C_q h\|u^n\|_{W^{1,q}} \quad\forall\, 2\le q<\infty.
\end{align}
Hence, by choosing $f^n=f_h^n$ in \eqref{BDF-Laplacian-1}, Theorem \ref{THM:LpLq-Laplacian} implies that the solutions of \eqref{BDF-FEM} and \eqref{BDF-Laplacian-1} satisfy the following error estimate: 
\begin{align*}
\big\|(P_hu^n-u_h^n )_{n=k}^N\big\|_{\ell^p(L^q)} 
&\le
C_{p,q} \big\|(P_hu^n-R_hu^n)_{n=k}^N\big\|_{\ell^p(L^q)} \\
&\le C_{p,q} h \| (u^n)_{n=k}^N \|_{\ell^p(W^{1,q})} \\
&\le C_{p,q} h \| (f_h^n)_{n=k}^N \|_{\ell^p(W^{-1,q})}  ,
\end{align*}
where the second to last inequality is due to \eqref{Ritz-Laplacian-Lq}, and the last inequality is due to the second result in Theorem \ref{THM:BDF-Laplacian}. By using the inverse inequality, we have
\begin{align*}
\big\|(P_hu^n-u_h^n )_{n=k}^N\big\|_{\ell^p(W^{1,q})} 
\le
Ch^{-1}
\big\|(P_hu^n-u_h^n )_{n=k}^N\big\|_{\ell^p(L^q)} 
\le C_{p,q}  \| (f_h^n)_{n=k}^N \|_{\ell^p(W^{-1,q})} .
\end{align*}
Then, using the triangle inequality, we obtain 
\begin{align*}
\big\|(u_h^n )_{n=k}^N\big\|_{\ell^p(W^{1,q})} 
\le
\big\|(P_hu^n)_{n=k}^N\big\|_{\ell^p(W^{1,q})} 
+\big\|(P_hu^n-u_h^n )_{n=k}^N\big\|_{\ell^p(W^{1,q})} 
\le
C_{p,q}  \| (f_h^n)_{n=k}^N \|_{\ell^p(W^{-1,q})} .
\end{align*}
The last inequality uses the $W^{1,q}$ stability of $P_h$ and the second result in Theorem \ref{THM:BDF-Laplacian}. 

Part 2: 
For $1<q\leq 2$, we express the solution of \eqref{BDF-FEM} as (see Lemma \ref{Lemma:Ehn-Linfty}) 
\begin{align*} 
u_h^n = \tau \sum_{j=k}^n E_h^{n-j} f_h^j  ,
\end{align*}
with $E_h^n$ is given by \eqref{expr-wh-1}. By considering the gradient of the equality above, we have 
\begin{align}\label{Lh-expr}
\nabla u_h^n =
( L_h \vec g )^n := \tau \sum_{j=k}^n \nabla E_h^{n-j} P_h \nabla\cdot g^j ,
\quad n=k,\dots,N,
\end{align}
where $g^j=a \nabla  A^{-1}f_h^j$ and $\vec g=(g^j)_{j=k}^N$, with $a=(a_{ij})$ denoting the $d\times d$ matrix of the diffusion coefficients (thus $\nabla\cdot g^j =f_h^j$). It suffices to prove that the operator $L_h$ is bounded on $\ell^p(L^q)$. To this end, we only need to prove the boundedness of its dual operator $L_h'$ on $\ell^{p'}(L^{q'})$. 
We define a discrete space-time inner product
$$
[\vec g,\vec\eta]_{\tau} = 
\tau \sum_{n=k}^N (g^n,\eta^n) \quad\mbox{for}\,\,\,
\vec g=(g^j)_{j=k}^N\,\,\,\mbox{and}\,\,\, \vec \eta=(\eta^j)_{j=k}^N .
$$
Then, using the definition of $L_h$ and integration by parts, we obtain $[L_h \vec g , \vec\eta]_{\tau} 
= [ \vec g , L_h' \vec\eta]_{\tau} $ with 
\begin{align*}
(L_h' \vec\eta)^j 
= \tau \sum_{n=j}^N \nabla E_h^{n-j} P_h \nabla\cdot \vec \eta^n .
\end{align*}
By a change of indices $j=N+k-j'$ and $n=N+k-n'$, we see that 
\begin{align}\label{Lh-dual-expr}
(L_h' \vec\eta)^{N+k-j'}
=  \tau \sum_{n'=k}^{j'} \nabla E_h^{j'-n'} P_h \nabla\cdot \eta^{N+k-n'} ,
\quad j'=k,\dots,N .
\end{align}
By comparing \eqref{Lh-expr} and \eqref{Lh-dual-expr}, we see that $w_h^{n}:=(L_h' \vec\eta)^{N+k-n}$ is the ``gradient'' of the solution to \eqref{BDF-FEM} with $f_h^n=P_h \nabla\cdot \vec \eta^{N+k-n}$. 
In Part 1, we have already shown that
$$
\|(w_h^n)_{n=k}^N\|_{\ell^{p'}(L^{q'})}
\le 
C\|(f_h^n)_{n=k}^{N}\|_{\ell^{p'}(W^{-1,q'})}
=
C\|(P_h \nabla\cdot \eta^{n})_{n=k}^{N}\|_{\ell^{p'}(W^{-1,q'})} , 
$$
where $2\le q'<\infty.$ for $1<q\le 2$. Since $P_h$ is symmetric and bounded on $W^{1,q}$, it follows that $P_h$ is also bounded on $W^{-1,q'}$ (as the dual space of $W^{1,q}$). Hence, the inequality above reduces to 
$$
\|(w_h^n)_{n=k}^N\|_{\ell^{p'}(L^{q'})}
\le 
C\|( \nabla\cdot \eta^{n})_{n=k}^N\|_{\ell^{p'}(W^{-1,q'})} 
\le
C\|( \eta^{n})_{n=k}^N\|_{\ell^{p'}(L^{q'})} .
$$
This prove the boundedness of $L_h'$ on $\ell^{p'}(L^{q'})$. By the duality between $\ell^{p'}(L^{q'})$ and $\ell^{p}(L^{q})$, the operator $L_h$ must be bounded on $\ell^p(L^{q})$. Using this boundedness of $L_h$, from \eqref{Lh-expr} we derive that 
\begin{align}\label{grad-uhn-fhn}
\|(\nabla u_h^n)_{n=k}^N\|_{\ell^{p}(L^q)}
\le
C\|(g^{n})_{n=k}^N\|_{\ell^{p}(L^{q})} 
=
C\|(a\nabla  A^{-1}f_h^n)_{n=k}^N\|_{\ell^{p}(L^{q})}   .
\end{align}
Using integration by parts and the symmetry of the operator $ A^{-1}$, we have 
\begin{align*}
|(a\nabla  A^{-1}f_h^n , v)|
= |( f_h^n ,  A^{-1} \nabla\cdot v)|
&\le
\| f_h^n \|_{W^{-1,q}} 
\|  A^{-1} \nabla\cdot (a v) \|_{W^{1,q'}} \\
&\le
C\| f_h^n \|_{W^{-1,q}} \| \nabla\cdot (a v) \|_{W^{-1,q'}} \\
&\le
C\| f_h^n \|_{W^{-1,q}} \| v \|_{L^{q'} } , 
\end{align*}
where the second to last inequality is exactly \eqref{W1q-Poisson}. The inequality above implies (via duality) 
$$
\|a\nabla  A^{-1}f_h^n\|_{L^{q}} 
\le
C\| f_h^n \|_{W^{-1,q}} . 
$$
Substituting this into \eqref{grad-uhn-fhn} yields the desired result, i.e.,
\begin{align}\label{uhn-W1q-fhn}
\|(u_h^n)_{n=k}^N\|_{\ell^{p}(W^{1,q})}
\le 
\C\|(f_h^n)_{n=k}^N\|_{\ell^{p}(W^{-1,q})}   .
\end{align}

Part 3: For both $2\le q<\infty$ and $1<q\le 2$, we have proved \eqref{uhn-W1q-fhn} in Part 1 and Part 2, respectively. Now, testing \eqref{BDF-FEM} by $v$, we obtain
\begin{align*}
\bigg|\bigg( \frac{1}{\tau} \sum_{j=0}^k \delta_j u_h^{n-j} ,v\bigg) \bigg| 
&=\bigg|\bigg( \frac{1}{\tau} \sum_{j=0}^k \delta_j u_h^{n-j} , P_hv\bigg) \bigg| \\
&=
\big| - \sum_{i,j=1}^d (a_{ij}\partial_j u_h^n,\partial_iP_hv) + (f_h^n, P_hv) \big| \\
&\le
C \big( \|\nabla u_h^n\|_{\ell^p(L^q)}+ \|f_h^n\|_{\ell^{p}(W^{-1,q})} \big)
\|v\|_{\ell^{p'}(W^{1,q'})} ,
\quad\forall\, v\in \ell^{p'}(W^{1,q'}).
\end{align*}
By the duality between $\ell^p(W^{-1,q})$ and $\ell^{p'}(W^{1,q'})$, the inequality above proves that 
$$
\bigg\|\bigg( \frac{1}{\tau} \sum_{j=0}^k \delta_j u_h^{n-j} \bigg)_{n=k}^N\bigg\|_{\ell^p(W^{-1,q})}
\le
C \big( \|\nabla u_h^n\|_{\ell^p(L^q)}+ \|f_h^n\|_{\ell^{p}(W^{-1,q})} \big) . 
$$
Then, using the equivalence relation \eqref{d_tau-equivlence}, we obtain 
$$
\|( d_\tau u_h^{n} )_{n=k}^N\|_{\ell^p(W^{-1,q})}
\le
C \big( \|\nabla u_h^n\|_{\ell^p(L^q)}+ \|f_h^n\|_{\ell^{p}(W^{-1,q})} \big) . 
$$
This completes the proof of Theorem \ref{THM:LpW1q-Laplacian}. 
\end{proof}

\begin{remark}\label{Remark:Lq-Ritz}
In the proof of Theorem \ref{THM:LpW1q-Laplacian} we have used \eqref{Ritz-Laplacian-Lq}. This can be proved by the following simple argument. 
For $2\leq q<\infty$, we define
$$
\left\{
\begin{aligned}
&- A v
=\operatorname{sign} 
\left(u-R_{h} u\right)\left|u-R_{h} u\right|^{q-1} &&\mbox{in}\,\,\,\varOmega,\\
&v=0 &&\mbox{on}\,\,\,\partial\varOmega ,
\end{aligned}
\right.
$$
and test the equation above by $u-R_{h} u$. Then we obtain, via integration by parts,
\begin{align*}
\|u-R_{h} u\|_{L^q}^q
=\sum_{i,j=1}^d (a_{ij}\partial_j  (u-R_{h} u),\partial_i v) 
&=\sum_{i,j=1}^d (a_{ij}\partial_j  (u-R_{h} u),\partial_i (v-P_hv)) \\ 
&\le
\|u-R_{h} u\|_{W^{1,q}} \, Ch \|v\|_{W^{2,q'}} .
\end{align*}
Since $1<q'\le 2$, by using the $W^{2,q'}$ estimate of elliptic equations in a convex polygon or polyhedron (see Lemma \ref{Lemma:W2q-aij}), we have
$$
\|v\|_{W^{2,q'}}
\le
C\|u-R_{h} u\|_{L^{(q-1)q'}}^{q-1}
=
C\|u-R_{h} u\|_{L^{q}}^{q-1} . 
$$
The last two estimates imply 
$$
\|u-R_{h} u\|_{L^q} 
\le
Ch \|u-R_{h} u\|_{W^{1,q}}  \quad\mbox{for}\,\,\, 2\le q<\infty. 
$$
This implies \eqref{Ritz-Laplacian-Lq}. It remains to prove the following lemma on the $W^{2,q'}$ estimate.  
\hfill\qed
\end{remark}

\begin{lemma}\label{Lemma:W2q-aij}
{\it 
In a convex polygon or polyhedron, there exists $q_0>2$ such that for any $1<q<q_0$ and $g\in L^{q}$, the equation 
\begin{align}\label{Elliptic-Eq-v-g}
\left\{
\begin{aligned}
& A v=g &&\mbox{in}\,\,\,\varOmega,\\
&v=0 &&\mbox{on}\,\,\,\partial\varOmega,
\end{aligned}
\right.
\end{align}
has a unique solution $v\in W^{2,q}$, which satisfies 
\begin{align}\label{W2q-convex}
\|v\|_{W^{2,q}}\le C\|g\|_{L^{q}} . 
\end{align} 
}
\end{lemma}
\begin{proof}
If the coefficients $a_{ij}$ are constants or smooth, then Lemma \ref{Lemma:W2q-aij} was already proved in \cite[Corollary 3.7 and 3.12]{Dauge1992} for some $q_0>2$. By using a perturbation argument, we prove that Lemma \ref{Lemma:W2q-aij} still holds for the same $q_0$ when $a_{ij}\in W^{1,d+\beta}$. 

When $a_{ij}\in W^{1,d+\beta}\hookrightarrow C^{1-\frac{d}{d+\beta}}$, the solution of \eqref{Elliptic-Eq-v-g} is given by 
$$
v(x)=\int_\varOmega G_A(x,y)g(y)\d y ,
$$
where $G_A(x,y)$ is the Green's function of the elliptic equation \eqref{Elliptic-Eq-v-g}, satisfying the following estimate in a convex domain (cf. \cite[Theorem 3.3]{GruterWidman1982}): 
$$
|\nabla_x G_A(x,y)|\le C|x-y|^{1-d}  .
$$
Hence, we have 
\begin{align*}
\|\nabla v\|_{L^{p,\infty}}
&=
\bigg\| \int_\varOmega \nabla_x G_A(x,y) g(y)\d y \bigg\| \\
&\le
C\bigg\| \int_\varOmega \frac{1}{|x-y|^{d-1}} g(y)\d y \bigg\| \\
&\le
C\bigg \| \frac{1}{|x|^{d-1}} \bigg\|_{L^{\frac{d}{d-1},\infty}} \|g\|_{L^{q}} \\
&\le
C \|g\|_{L^{q}} 
\qquad\mbox{with}\quad
\frac{1}{p} 
=\frac{1}{q} - \frac{1}{d} < \frac{1}{q} - \frac{1}{d+\beta} , 
\end{align*}
where the second to last inequality is Young's inequality of weak type, and $L^{\frac{d}{d-1},\infty}$ is the weak-type $L^{\frac{d}{d-1}}$ space; see \cite[Theorem 1.2.13]{Grafakos2008}. The inequality above implies that 
\begin{align}\label{grad-v-Lq-star}
\nabla v\in L^{p,\infty}\hookrightarrow L^{q_*}\quad\mbox{for}\quad 
\frac{1}{q_*}  := \frac{1}{q} - \frac{1}{d+\beta} \quad\mbox{(because $q_*<p$)} .
\end{align}

For any point $x_0\in \varOmega$, we consider a smooth cut-off function $\omega_\varepsilon$ with the following properties:
\begin{align}\label{properties-omega}
\omega_\varepsilon(x)= 
\left\{
\begin{aligned}
&1 && |x-x_0|<\varepsilon ,\\
&0 && |x-x_0|\ge 2\varepsilon, 
\end{aligned}
\right.
\quad\mbox{and}\quad
\|\nabla^m\omega_\varepsilon\|_{L^\infty}\le C\varepsilon^{-m} 
\,\,\,\mbox{for}\,\,\,m\ge 0 .
\end{align}
Since $a_{ij}\in W^{1,d+\beta}\hookrightarrow C^{1-\frac{d}{d+\beta}}$, it follows that 
\begin{align}\label{aijx-aijx0}
|a_{ij}(x)-a_{ij}(x_0)|
\le C\varepsilon^{1-\frac{d}{d+\beta}}
\quad\mbox{on the support of $\omega_\varepsilon$}. 
\end{align}
Then multiplying \eqref{Elliptic-Eq-v-g} by $\omega_\varepsilon$ yields 
\begin{align}\label{Eq-L}
\left\{
\begin{aligned}
& L (\omega_\varepsilon v) 
= \omega_\varepsilon g 
+  \sum_{i,j=1}^d   ( a_{ij} v \partial_i \partial_{j} \omega_\varepsilon
+ a_{ij}\partial_j  v \partial_i \omega_\varepsilon
- \partial_i a_{ij} \omega_\varepsilon \partial_j v )
&&\mbox{in}\,\,\,\varOmega,\\
&\omega_\varepsilon v=0 &&\mbox{on}\,\,\,\partial\varOmega,
\end{aligned}
\right.
\end{align}
with 
$$
L
=
\sum_{i,j=1}^d a_{ij}(x_0) \partial_i \partial_j 
+\sum_{i,j=1}^d (a_{ij} -a_{ij}(x_0)) \partial_i \partial_j 
=: L_0 + B .
$$ 
In view of \eqref{aijx-aijx0}, we can choose a sufficiently small $\varepsilon$ (smaller than some constant depending on $\|a_{ij}\|_{W^{1,d+\beta}}$) such that $L$ becomes a small perturbation of $L_0$ as an operator from $W^{2,q}$ to $L^q$. Since the operator $L_0:W^{2,q}\rightarrow L^q$ (with constant coefficients) is invertible when $1<q<q_0$, it follows that $\omega_\varepsilon v\in W^{2,q}$ if the right-hand side of \eqref{Eq-L} is in $L^q$ when $1<q<q_0$. Indeed, the right-hand side is in $L^q$ because 
\begin{align*}
\|a_{ij} v \partial_{ij} \omega_\varepsilon
+ a_{ij}\partial_j  v \partial_i \omega_\varepsilon-\partial_ia_{ij} \omega_\varepsilon \partial_jv\|_{L^q}
\le
C\varepsilon^{-2} \|a_{ij}\|_{W^{1,d+\beta}} \|v\|_{W^{1,q_*}} 
\le
C\varepsilon^{-2} ,
\end{align*}
where the last inequality is due to \eqref{grad-v-Lq-star}, and the second to last inequality is due to the properties of the cut-off function $\omega_\varepsilon$. This proves that for any $x_0\in\varOmega$ there is a neighborhood $B_\varepsilon(x_0)\cap \varOmega$ (of constant radius $\varepsilon$) on which the solution $v$ is in $W^{2,q}$. 
\end{proof}

\section{Proof of Theorem \ref{Resolvent-Laplacian}}
\label{section:proof}

In Section \ref{section:main}, we have used the results of Theorem \ref{Resolvent-Laplacian} to establish maximal $L^p$-regularity of multistep fully discrete FEMs. In this section, we prove Theorem \ref{Resolvent-Laplacian} by formulating it into a form which can be proved in the same way as \cite[Proof of Corollary 2.1]{Li2019}. Then we only sketch the proof by following the outline of \cite[Proof of Corollary 2.1]{Li2019} and highlighting the difference from \cite[Proof of Corollary 2.1]{Li2019}. 

\subsection{Notation}
We use the same notation as \cite[Section 3]{Li2019}, including the notation of function spaces, local approximation properties of the finite element spaces,  delta function, regularized delta functions, Green's function, regularized Green's function, and Dyadic decomposition of the domain. These notations will not be duplicated in the current paper. The only changes of notation in the current paper are: \vspace{-1pt}
\begin{enumerate}
\item
The dimension of space is denoted by $d$ in this article (and $N$ in \cite{Li2019}).\vspace{-5pt}

\item
The elliptic operator is $A=\sum_{i,j=1}^d\partial_i(a_{ij}\partial_j)$ in this article (and Laplacian $\Delta$ in \cite{Li2019}).
\end{enumerate}
We denote by 
\begin{align}
(u,v)=\int_\varOmega u(x)v(x)\d x 
\quad\mbox{and}\quad
\langle u,v\rangle=\int_\varOmega u(x) \overline{v(x)} \d x 
\quad 
\forall\, u\in L^q,\,\,\,v\in L^{q'},
\end{align}
the real and complex pairing between two complex-valued functions on $\varOmega$, respectively. This notation is consistent with \eqref{def-inner-product} for real-valued functions. 

\subsection{Complex Green's function and its regularized approximation}

We denote by $E(t)=e^{t A}$ the semigroup generated by $ A$ on $L^q$, and denote by $E_h(t)=e^{tA_h}$ the semigroup generated by $A_h$ on $S_h$. 
Since the semigroup $E(t)$ has a bounded analytic extension $E(z):L^q\rightarrow L^q$ to $z\in \Sigma_{\theta}$ for all $\theta\in (-\frac{\pi}{2},\frac{\pi}{2})$ (cf. \cite[Theorem 2.4]{Ouhabaz1995}), we can define 
$$
E^\theta(t):=E(te^{\i\theta}) 
\quad\mbox{for any angle $\theta\in(-\frac{\pi}{2},\frac{\pi}{2})$.}
$$
Then the function 
$$u(t)=E^\theta(t)u_0+ \int_0^t E^\theta(t-s)f(s)\d s $$ is the solution of the complex-valued parabolic problem 
\begin{align} 
\label{PDE1-theta} 
\left\{\begin{array}{ll}
\displaystyle
\partial_t u - e^{\i\theta}  A u = f
&\mbox{in}\,\,\, \R_+\times \varOmega,\\[3pt]
u=0 
&\mbox{on}\,\, \R_+\times\partial\varOmega ,\\[3pt]
u|_{t=0}=u_0
&\mbox{in}\,\,\, \varOmega .
\end{array}\right. 
\end{align} 

It is known that the Green's function $G(t,x,y)$ defined by
\begin{align}\label{GFdef}
\left\{
\begin{aligned}
&\partial_t G(t,x,y) - \sum_{i,j=1}^d \frac{\partial}{\partial x_i} \bigg( a_{ij}(x) \frac{\partial G(t,x,y) }{\partial x_j} \bigg)
= 0 &&\mbox{in}\,\,\, \varOmega\times(0,T], \\
&G(t,x,y) =0 &&\mbox{on}\,\,\,\partial\varOmega\times(0,T], \\[5pt]
&G(0,x,y)=\delta_y(x)&&\mbox{in}\,\,\,\varOmega .
\end{aligned}
\right.
\end{align} 
is symmetric with respect to $x$ and $y$, and has an analytic extension $G(z,x,y)$ for $z\in \Sigma_{\varphi}$, 
satisfying the following Gaussian estimate
(cf. \cite[p.\ 103]{Davis1989} or \cite[Proposition 2.3]{Ouhabaz1995}): 
\begin{equation}
|G(z,x,y)|\leq C_\varphi|z|^{-\frac{d}{2}}
e^{-\frac{|x-y|^2}{C_\varphi|z|}}, 
\quad \forall\, z\in \varSigma_{\varphi},\,\,\forall\, x,y\in\varOmega ,
\quad\forall\,\varphi\in(0,\pi/2), 
\label{GKernelE0}
\end{equation}
where the constant $C_\varphi$ depends only on $\varphi$. The Cauchy integral formula implies that 
\begin{equation}
\partial_z^kG(z,x,y)
=\frac{k!}{2\pi i}
\int_{|\zeta-z|=\frac12|z|\sin(\varphi)}\,\, \frac{G(\zeta,x,y)}{(\zeta-z)^{k+1}}\d \zeta \qquad\forall\,z\in\Sigma_\varphi,
\end{equation}
which further yields the following Gaussian pointwise estimate for the time derivatives of Green's function: 
\begin{align}
&|\partial_z^kG(z,x,y)|\leq \frac{C_{\varphi,k}}{z^{k+d/2}} e^{-\frac{|x-y|^2}{C_{\varphi,k}|z|}}, &&
\forall\, x,x_0\in\varOmega,\,\,\, \forall\, z\in\Sigma_\varphi, \,\, k=0,1,2,\dots \label{GausEst1}
\end{align} 
Since $G(t,x,y)=G(t,y,x)$ (symmetry of Green's function) for $t>0$, it follows that their analytic extensions are also equal, i.e.,
\begin{align}\label{symmetric-G}
G(z,x,y)=G(z,y,x)\quad\forall\,x,y\in\varOmega,\,\,\,\forall\, z\in\Sigma_{\varphi}  .
\end{align} 

It is straightforward to verify that $G^\theta(t,x,y):=G(te^{\i\theta},x,y)$ is the solution of the complex-valued parabolic equation 
\begin{align}\label{GFdef-complex}
\left\{\begin{array}{ll}
\partial_tG^\theta(\cdot,\cdot\, ,y)- e^{\i\theta}  A G^\theta(\cdot,\cdot\, ,y)=0
&\mbox{in}\,\,\, \R_+ \times \varOmega,\\
G^\theta(\cdot,\cdot\, ,y) = 0
&\mbox{on}\,\,\, \R_+\times \partial\varOmega ,\\
G^\theta(0,\cdot,y)= \delta_{y}
&\mbox{in}\,\,\,\varOmega .
\end{array}\right.
\end{align}
The estimate \eqref{GausEst1} implies that 
\begin{align}
&|\partial_t^kG^\theta(t,x,y)|\leq 
\frac{C_{\theta,k}}{t^{k+d/2}} e^{-\frac{|x-y|^2}{C_{\theta,k} t}}, &&
\forall\, x,y\in\varOmega,\,\,\, \forall\, t>0, \,\, k=0,1,2,\dots \label{GausEst-phi-1}
\end{align} 
where the constant $C_{\theta,k}$ would be bounded when $\theta$ is bounded away from $\pm\frac\pi2$. 
%

We define the regularized Green's function $\Gamma(t,x, y)$ by
\begin{align}\label{RGFdef0}
\left\{
\begin{aligned}
&\partial_t \Gamma(t,x,y) - \sum_{i,j=1}^d \frac{\partial}{\partial x_i} \bigg( a_{ij}(x) \frac{\partial \Gamma(t,x,y) }{\partial x_j} \bigg)
= 0 &&\mbox{in}\,\,\, \varOmega\times(0,T], \\
&\Gamma(t,x,y) =0 &&\mbox{on}\,\,\,\partial\varOmega\times(0,T], \\[5pt]
&\Gamma(0,x,y)=\widetilde\delta_y(x)&&\mbox{in}\,\,\,\varOmega ,
\end{aligned}
\right.
\end{align}
where $\widetilde\delta_y$ denotes the regularized delta function defined in \cite[Section 3.3]{Li2019}. Similarly as the complex Green's function, we define $\Gamma^\theta(t,x, y)=\Gamma(te^{\i\theta},x, y)$, which is the solution of
\begin{align}\label{RGFdef}
\left\{
\begin{aligned}
&\partial_t \Gamma^\theta(t,x,y) - e^{\i\theta} \sum_{i,j=1}^d \frac{\partial}{\partial x_i} \bigg( a_{ij}(x) \frac{\partial \Gamma^\theta(t,x,y) }{\partial x_j} \bigg)
= 0 &&\mbox{in}\,\,\, \varOmega\times(0,T], \\
&\Gamma^\theta(t,x,y) =0 &&\mbox{on}\,\,\,\partial\varOmega\times(0,T], \\[5pt]
&\Gamma^\theta(0,x,y)=\widetilde\delta_y(x)&&\mbox{in}\,\,\,\varOmega ,
\end{aligned}
\right.
\end{align}
and can be represented by 
\begin{align}\label{expr-Gamma}
&\Gamma^\theta(t,x,y)=\int_\varOmega  G^\theta(t,x',x)\widetilde\delta_{y}(x')\d x'=\int_\varOmega G^\theta(t,x,x')\widetilde\delta_{y}(x')\d x'  . 
\end{align}
From \eqref{expr-Gamma} and \eqref{GausEst1} we can derive the following pointwise estimate:
\begin{align}
&|\partial_t^k\Gamma^\theta(t,x,y)|\leq \frac{C_{\theta,k}}{t^{k+d/2}} e^{-\frac{|x-y|^2}{C_{\theta,k} t}},  
\quad k=0,1,2,\dots  
\label{GausEstGamma}
\end{align} 
for $x,y\in\varOmega$ and $t>0$ such that $\max(|x-y|,\sqrt{t})\ge 2h$. 

Similarly, for the discrete Green's function $\Gamma_h(t, \cdot , y)\in S_h$ defined by
\begin{align}\label{EqGammh-0}
\left\{\begin{aligned} 
&(\partial_t\Gamma_h (t,\cdot,y),v_h) + \sum_{i,j=1}^d (\partial_j\Gamma_h(t,\cdot,y),\partial_i v_h)=0 
&&\forall \, v_h\in S_h,\,\, \forall\, t>0 ,\\[5pt]
&\Gamma_h(0,\cdot,y)=  \delta_{h,y} ,
\end{aligned}\right.
\end{align}
we define the complex-valued finite element function $\Gamma^\theta_h(t,x, y)=\Gamma_h(te^{\i\theta},x, y)$, which is the solution of 
\begin{align}\label{EqGammh}
\left\{\begin{aligned} 
&(\partial_t\Gamma_h^\theta(t,\cdot,y),v_h) + e^{\i\theta} \sum_{i,j=1}^d (\partial_j\Gamma_h^\theta(t,\cdot,y),\partial_i v_h)=0 
&&\forall \, v_h\in S_h,\,\, \forall\, t>0 ,\\[5pt]
&\Gamma_h^\theta(0,\cdot,y)=  \delta_{h,y} .
\end{aligned}\right.
\end{align}

\subsection{A key lemma}

We fix an arbitrary angle $\theta\in (-\frac{\pi}{2},\frac{\pi}{2})$ and consider the finite element solution of the following semidiscrete problem: 
\begin{align} 
\label{FEM1-theta} 
\left\{\begin{array}{ll}
\displaystyle
(\partial_t u_h(t),v_h) + e^{\i\theta} \sum_{i,j=1}^d (\partial_j  u_h(t),\partial_i v_h) = (f_h(t) ,v_h)
\quad \forall\, v_h\in S_h,\,\,\forall\, t>0, \\[8pt]
u_h(0)=u_{0,h} , 
\end{array}\right. 
\end{align} 
which is the finite element approximation of \eqref{PDE1-theta}. 
The solution of \eqref{FEM1-theta} can be written as 
$$u_h(t)=E_h^\theta(t)u_{0,h} + \int_0^t E_h^\theta(t-s)f_h(s)\d s ,$$ 
where $E_h^\theta(t)$ is the semigroup  generated by $e^{\i\theta}A_h$. 
By using the Green's functions defined in \eqref{GFdef-complex} and \eqref{EqGammh}, the solutions of \eqref{PDE1-theta} and \eqref{FEM1-theta} can be represented by 
\begin{align}
&u(t,y)=\int_\varOmega  G^\theta(t,x,y)u_0(x)\d x + \int_0^t\int_\varOmega G^\theta(t-s,x,y)f(s,x)\d x\d s ,\\
&u_h(t,y)=\int_\varOmega  \Gamma_h^\theta(t,x,y)u_{0,h}(x)\d x + \int_0^t\int_\varOmega \Gamma_h^\theta(t-s,x,y)f(s,x)\d x\d s , \label{uh-Gammah} 
\end{align}
and
\begin{align}
&(E^\theta(t)v)(y)=\int_\varOmega  G^\theta(t,x,y) v(x)\d x  , 
\qquad
(E_h^\theta(t)v_h)(y)=\int_\varOmega  \Gamma_h^\theta(t,x,y)v_{h}(x)\d x .  \label{semigroup-Green-discr}
\end{align}
Clearly, the operator $E_h^\theta(t)P_h:L^q \rightarrow S_h$ is an extension of $E_h^\theta(t):S_h\rightarrow S_h$ to the domain $L^q$. Since $\Gamma_h^\theta(t,x,y)$ is a finite element function in $x$, it follows that 
\begin{align}
&(E_h^\theta(t) P_h v) (y) 
=\int_\varOmega  \Gamma_h^\theta(t,x,y) P_hv (x) \d x 
=\int_\varOmega  \Gamma_h^\theta(t,x,y)v (x) \d x 
\quad\forall\, v\in L^q. 
\end{align}

We shall prove the following result, which is the key to prove Theorem \ref{Resolvent-Laplacian}.

\begin{lemma}\label{MainLemma1}
{\it 
For any $\theta\in(0,\frac\pi2)$, the following estimates hold: 
\begin{align}
&\sup_{t>0}\,(\|E_h^\theta(t)v_h\|_{L^{q}}
+t\|\partial_tE_h^\theta(t)v_h\|_{L^{q}})
\leq C_\theta\|v_h\|_{L^{q}}  
&&\forall\,\, v_h\in 
S_h ,
&&\forall\,\, 1\leq q\leq\infty , \label{analyticity} \\ 
&\left\|\sup_{t>0}|E_h^\theta(t)P_h|\, |v| \right\|_{L^{q}}
\leq C_{\theta,q}\|v\|_{L^{q}} 
&& \forall\,v\in L^q, &&\forall\, 1<q\le \infty , \label{MEgodic2}
\end{align}
where 
\begin{align}
(|E_h^\theta(t)|v) (y):=\int_\varOmega |\Gamma_h^\theta(t,x,y)| v(x)\d x ,\quad\forall\, v\in L^q ,
\end{align}
and the constants $C$ and $C_q$ are independent of $h$ and $T$ (bounded when $\theta$ is away from $\pm \frac\pi2$). 
}
\end{lemma}

\subsection{Proof of Theorem \ref{Resolvent-Laplacian} based on Lemma \ref{MainLemma1}}

\indent\quad\,\,\,{\it Proof of (1)}: 
By the theory of analytic semigroups \cite[Theorem 3.7.19]{ABHN}, inequality \eqref{analyticity} implies the existence of an angle $\varphi\in(0,\pi/2)$, 
such that the semigroup $\{E_h^\theta(t)\}_{t>0}$ extends to be a
bounded analytic semigroup $\{E_h^\theta(z)\}_{z\in\Sigma_{\varphi}}$ in the sector $\Sigma_{\varphi}$, satisfying
\begin{align} \label{analyticity-varphi}
&\sup_{z\in \Sigma_{\varphi}}\,(\|E_h^\theta(z)v_h\|_{L^{q}}
+|z| \|\partial_zE_h^\theta(z)v_h\|_{L^{q}})
\leq C_\theta^* \|v_h\|_{L^{q}}  
&&\forall\,\, v_h\in 
S_h ,
&&\forall\,\, 1\leq q\leq\infty . 
\end{align}
From the proof of \cite[Theorem 3.7.19]{ABHN} we see that both $\varphi$ and $C_\theta^*$ in \eqref{analyticity-varphi} depend only on the constant $C_\theta$ in \eqref{analyticity} (thus independent of $h$ and $q$).
Then, \eqref{analyticity-varphi} and \cite[Theorem 3.7.11]{ABHN} imply the following result:
\begin{equation}\label{resovlent-theta} 
\mbox{$\big\{z(z - e^{\i\theta}A_h)^{-1} \,:\, z\in\Sigma_{\frac\pi2} \big\}$\, is bounded on $L^q$ for $1\leq q\leq\infty$} ,
\end{equation}
with a bound depending only $C_\theta^*$ (thus independent of $h$). Rewriting the operator $z(z - e^{\i\theta}A_h)^{-1}$ as $e^{-\i\theta}z(e^{-\i\theta}z - A_h)^{-1}$ and replacing $\theta$ by $\pm\theta$ in \eqref{angle-condition-phi}, we obtain the following result (after adding a bounded operator $P_h$): 
\begin{align*}
&\mbox{$\left\{z (z - A_h)^{-1}P_h \,:\, z\in\mathbb{C},\,\, -\theta-\frac{\pi}{2} \le \arg(z) \le -\theta + \frac{\pi}{2} \right\}$} \\[5pt]
\mbox{and}\quad
&\mbox{$\left\{z (z - A_h)^{-1}P_h \,:\, z\in\mathbb{C},\,\, \theta-\frac{\pi}{2} \le \arg(z) \le \theta+\frac{\pi}{2} \right\}$} \quad\mbox{are both bounded on $L^q$},
\end{align*}
for $1\le q\le \infty$, with a bound independent of $h$. Since the union of the two bounded sets above is also bounded, we obtain the first result of Theorem \ref{Resolvent-Laplacian}. 

{\it Proof of (2)}: 
Estimate \eqref{MEgodic2} immediately implies the maximal ergodic estimate
\begin{align}
&\bigg\|\sup_{t>0}\frac{1}{t}\int_0^t |E_h^\theta(s)P_h| |v| \d s \bigg\|_{L^{q}}
\leq C_{\theta,q}\|v\|_{L^{q}} ,
&& \forall\,\, 1<q\le\infty . \label{MEgodic3}
\end{align}

For $q\in(1,2]$, according to \cite[Lemma 4.c]{Weis2001-1}, \eqref{MEgodic3} implies the $R$-boundedness of the semigroup $\{E_h^\theta(z)P_h\}_{z\in\Sigma_{\sigma}}$ on $L^q$, with an angle $\sigma=\varphi q/4$ (thus independent of $h$); from \cite[Proofs of Lemma 4.c and Proposition 4.b]{Weis2001-1}, we see that the $R$-bound depends only on the constants $C_\theta^*$ and $C_{\theta,q}$ in \eqref{analyticity-varphi}--\eqref{MEgodic3} (thus independent of $h$).  

For $q\in[2,\infty)$, we use the fact that the dual operator $E_h^\theta(z)P_h$ is itself, i.e.,
\begin{align}\label{symmetry-Eh-theta}
(E_h^\theta(z)P_hu_0 , w_0 ) = 
( u_0 , E_h^\theta(z)P_h w_0) \quad\forall\, u_0\in L^q,\,\,\, w_0\in L^{q'}.
\end{align} 
This can be seen in the following way. The function $u_h(t)=E_h^\theta(z)P_hu_0$ is a solution of \eqref{FEM1-theta} with $u_{0,h}=P_hu_0$ and $f_h=0$, and the function $w_h(T-t)= (E_h^\theta(z)P_hw_0)(T-t)$ is a solution of the following backward equation: 
\begin{align} 
\label{FEM1-theta-v} 
\left\{\begin{array}{ll}
\displaystyle
(\partial_t w_h(t),v_h) - e^{\i\theta} \sum_{i,j=1}^d (\partial_j w_h(t),\partial_i v_h) = 0
\quad \forall\, v_h\in S_h,\,\,\forall\, t\in[0,T), \\[8pt]
w_h(T)=P_hw_0 . 
\end{array}\right. 
\end{align} 
Hence, substituting $f_h=0$ and $v_h=w_h$ into \eqref{FEM1-theta} and using integration by parts in time, we obtain \eqref{symmetry-Eh-theta}. 
By using the symmetry \eqref{symmetry-Eh-theta}, we see that for $q\in[2,\infty)$ the dual operators $E_h^\theta(z)'=E_h^\theta(z)$, $z\in\Sigma_{\sigma}$, are $R$-bounded on $L^{q'}$ with angle $\sigma=\varphi q'/4$. This implies that the operators $E_h^\theta(z)$, $z\in\Sigma_{\sigma}$, are $R$-bounded on $L^q$ (cf. \cite[Remark 4.b. (ii) with $p=2$]{Weis2001-1}). 

Therefore, for $1<q<\infty$, the following result holds: 
\begin{align}\label{Eh-theta-z}
\mbox{$E_h^\theta(t)$ has a $R$-bounded analytic extension $E_h^\theta(z)$ for $z\in\Sigma_{\sigma}$}, 
\end{align}
where $\sigma=\varphi_\theta \min(q',q)/4$ and the $R$-bound is independent of $h$. This proves the second result of Theorem \ref{Resolvent-Laplacian}. 

{\it Proof of (3)}: 
Since 
$$
z(z - e^{\i\theta}A_h)^{-1}P_h=  z \int_0^\infty e^{-z t}E_h^\theta(t) P_h \d t ,
$$ 
it follows that (cf. \cite[Theorem 2.10]{Weis2001-2})  
\begin{equation}\label{angle-condition-phi} 
\mbox{$\left\{z(z - e^{\i\theta}A_h)^{-1}P_h \,:\, z\in\mathbb{C},\,\,|\arg(z)|\le \frac{\pi}{2} \right\}$\, is $R$-bounded on $L^q$} ,
\end{equation} 
with an $R$-bound depending only on $\sigma$ and the $R$-bound of $E_h^\theta(z)$ (thus independent of $h$).
Rewriting $z(z - e^{\i\theta}A_h)^{-1}P_h$ as $e^{-\i\theta}z(e^{-\i\theta}z - A_h)^{-1}P_h$ and replacing $\theta$ by $\pm\theta$ in \eqref{angle-condition-phi}, we obtain the following result: 
\begin{align*}
&\mbox{$\left\{z (z - A_h)^{-1}P_h \,:\, z\in\mathbb{C},\,\, -\frac{\pi}{2}-\theta \le \arg(z) \le \frac{\pi}{2}-\theta \right\}$} \\[5pt]
\mbox{and}\quad
&\mbox{$\left\{z (z - A_h)^{-1}P_h \,:\, z\in\mathbb{C},\,\, -\theta+\frac{\pi}{2} \le \arg(z) \le \theta+\frac{\pi}{2} \right\}$} \quad\mbox{are both $R$-bounded},
\end{align*}
with an $R$-bound independent of $h$.
Since the union of the two $R$-bounded sets above is also $R$-bounded,
we obtain the third result of Theorem \ref{Resolvent-Laplacian}. 

It remains to prove Lemma \ref{MainLemma1}.

\subsection{Proof of Lemma \ref{MainLemma1}}


Lemma \ref{MainLemma1} can be proved in the same way as \cite[Theorem 2.1]{Li2019} by using the following lemma. Hence, the proof of Lemma \ref{LemGm2} is omitted. 

\begin{lemma}\label{LemGm2}
{\it
The functions $\Gamma_h^\theta(t,x,x_0)$, $\Gamma^\theta(t,x,x_0)$ and 
$F^\theta(t,x,x_0):=\Gamma_h^\theta(t,x,x_0)-\Gamma^\theta(t,x,x_0)$ satisfy 
\begin{align}
&\sup_{x_0\in\varOmega}\sup_{t\in(0,\infty)}\,(\|\Gamma_h^\theta(t,\cdot, x_0)\|_{L^1(\varOmega)}  
+t\|\partial_{t}\Gamma_h^\theta(t,\cdot, x_0)\|_{L^1(\varOmega)} ) \leq C , 
\label{L1Gammh}\\ 
&\sup_{x_0\in\varOmega}\sup_{t\in(0,\infty)}\,(\|\Gamma^\theta(t,\cdot, x_0)\|_{L^1(\varOmega)}  
+t\|\partial_{t}\Gamma^\theta(t,\cdot, x_0)\|_{L^1(\varOmega)} ) \leq C , 
\label{L1Gammh-2}\\ 
&\sup_{x_0\in\varOmega} \big( \|\partial_t F^\theta(\cdot,\cdot ,x_0)\|_{L^1((0,\infty)\times\varOmega)}  
+\|t\partial_{tt}F^\theta(\cdot,\cdot , x_0)\|_{L^1((0,\infty)\times\varOmega)} \big) \leq C ,
\label{L1Ft}\\
&
\sup_{x_0\in\varOmega} \|\partial_t\Gamma_h^\theta(t,\cdot, x_0)\|_{L^1}\leq Ce^{-\lambda_0t} , \qquad\forall\, t\ge 1, \label{L1Gammatx0}
\end{align} 
where the constants $C$ and $\lambda_0$ are independent of $h$. 
}
\end{lemma}

Lemma \ref{LemGm2} can be proved in the same way as \cite[Lemma 4.4]{Li2019} by using the following two lemmas. Hence, the proof of Lemma \ref{LemGm2} is omitted. 

\begin{lemma}[Local energy error estimate for parabolic equations]\label{LocEEst} 
{\it 
If $\phi\in L^2(0,T;H^1_0)\cap H^1(0,T;L^2)$ and 
$\phi_h\in H^1(0,T; S_h)$ satisfy the equation 
\begin{align}  
(\partial_t(\phi-\phi_h),\chi_h)
+ e^{\i\theta} \sum_{i,j=1}^d (\partial_j  (\phi-\phi_h),\partial_i \chi_h) =0, 
\quad\forall\, \chi_h\in  S_h, \,\, \mbox{a.e.}\,\, t>0, 
\label{p32}  
\end{align} 
with $\phi(0)=0$ in $\varOmega_j''$. Then we have 
\begin{align} 
&\vertiii{\partial_t(\phi-\phi_h)}_{Q_j} 
+ d_j^{-1}\vertiii{\phi-\phi_h}_{1,Q_j} \notag\\
& \leq 
C\epsilon^{-3}\big(I_j(\phi_{h}(0))+X_j(I_h\phi-\phi) 
+d_j^{-2}\vertiii{\phi-\phi_h}_{Q_j'}\big) \notag\\
&\quad 
+(Ch^{1/2}d_j^{-1/2}+C\epsilon^{-1}hd_j^{-1}+\epsilon) \big(\vertiii{\partial_t(\phi-\phi_h)}_{Q_j'}
+d_j^{-1}\vertiii{\phi-\phi_h}_{1,Q_j'}\big), 
\label{LocEngErr}  
\end{align} 
where
\begin{align*}
&I_j(\phi_{h}(0))=\|\phi_{h}(0)\|_{1,\varOmega_j'} +d_j^{-1}\|\phi_{h}(0)\|_{\varOmega_j'} \, ,\\[5pt]
&X_j(I_h\phi-\phi)=d_j\vertiii{\partial_t(I_h\phi-\phi)}_{1,Q_j'}
+\vertiii{\partial_t(I_h\phi-\phi)}_{Q_j'} \notag\\
&\qquad\qquad\qquad\,\, 
+d_j^{-1}\vertiii{I_h\phi-\phi}_{1,Q_j'}+
d_j^{-2}\vertiii{I_h\phi-\phi}_{Q_j'}  \, ,
\end{align*}
where $\epsilon\in(0,1)$ is an arbitrary positive constant, and 
the positive constant $C$ is independent of $h$, $j$ and $C_*$; the norms $\ii\cdot\ii_{k,Q_j'}$ and $\ii\cdot\ii_{k,\varOmega_j'}$ are defined in 
{\rm\cite[(3.23)]{Li2019}}. 
}
\end{lemma}

\begin{lemma}[Local estimates of the complex Green's function]\label{GFEst1}
{\it There exists $\alpha\in (\frac{1}{2},1]$ and $C>0$, independent of $h$ and $x_0$, such that the complex Green's function $G^\theta$ defined in \eqref{GFdef-complex} and the complex regularized Green's function $\Gamma^\theta$ defined in \eqref{RGFdef}
satisfy the following estimates:
\begin{align}
&d_j^{\frac{d}{2}-4-\alpha}\|\Gamma^\theta(\cdot,\cdot,x_0)\|_{L^\infty(Q_j(x_0))}
+d_j^{-4-\alpha}\ii\nabla\Gamma^\theta(\cdot,\cdot,x_0)\ii_{L^2(Q_j(x_0))}\notag\\
&\quad
+d_j^{-4} \ii \Gamma^\theta(\cdot,\cdot,x_0)\ii_{L^2H^{1+\alpha}(Q_j(x_0))}  +d_j^{-2} \ii \partial_{t} 
\Gamma^\theta(\cdot,\cdot,x_0) \ii_{L^2H^{1+\alpha}(Q_j(x_0))} \notag\\
&\quad 
+ \ii \partial_{tt}\Gamma^\theta(\cdot,\cdot,x_0) \ii_{L^2H^{1+\alpha}(Q_j(x_0))} 
\le Cd_j^{-\frac{d}{2}-4-\alpha}, \label{GFest01}\\[5pt] 
&\|G^\theta(\cdot,\cdot,x_0) \|_{L^{\infty}H^{1+\alpha}(\cup_{k\leq
j}Q_k(x_0))} 
+d_j^2\|\partial_tG^\theta(\cdot,\cdot,x_0) \|_{L^{\infty}H^{1+\alpha}(\cup_{k\leq
j}Q_k(x_0))} 
\le Cd_j^{-\frac{d}{2}-1-\alpha} \label{GFest03} .
\end{align}
}
\end{lemma}

The only difference between Lemma \ref{LemGm2} and \cite[Lemma 5.1]{Li2019} is the presence of $e^{\i\theta}$ in \eqref{p32}. This does not affect the proof of Lemma \ref{LocEEst}. Hence, the proof of Lemma \ref{LemGm2} is also omitted. 
Lemma \ref{GFEst1} can be proved in the same way as \cite[Lemma 4.1]{Li2019} based on the following regularity result. 

\begin{lemma}\label{RegPoiss}
{\it There exists a positive constant $\alpha\in(\frac{1}{2},1]$ such that the solution of the elliptic equation 
\begin{align} \label{H1+alpha-regularlity}
\left\{\begin{array}{ll} 
A v = g 
&\mbox{in}\,\,\,\varOmega, \\[5pt] 
\varphi=0 
&\mbox{on}\,\,\,\partial\varOmega, 
\end{array}\right.  
\end{align} 
satisfies 
\begin{align*} 
\|v\|_{H^{1+\alpha}} 
\leq C\|g\|_{H^{-1+\alpha}} . 
\end{align*} 
}
\end{lemma}
\begin{proof}
When $a_{ij}=a_{ji}=$ constants, Lemma \ref{RegPoiss} holds as explained in \cite[Lemma 4.1]{Li2019}. When $a_{ij}=a_{ji}\in W^{1,d+\beta}$, it can be proved by using a perturbation argument as shown below. 

In fact, for any $x_0\in\overline\varOmega$ we can introduce a smooth cut-off function $\omega_\varepsilon$ satisfying \eqref{properties-omega}--\eqref{aijx-aijx0}, and reformulate equation \eqref{H1+alpha-regularlity} into 
\begin{align}\label{Eq-L-2}
\left\{
\begin{aligned}
&\sum_{i,j=1}^d \partial_j \big[ a_{ij}(x_0) \partial_i (\omega_\varepsilon v) \big]
= g_{\varepsilon}
&&\mbox{in}\,\,\,\varOmega,\\
&\omega_\varepsilon v=0 &&\mbox{on}\,\,\,\partial\varOmega,
\end{aligned}
\right.
\end{align}
with 
$$
g_{\varepsilon}
= 
\omega_\varepsilon g 
+ \sum_{i,j=1}^d  \big[ 2\partial_j (a_{ij} v) \partial_i \omega_\varepsilon 
-\partial_j a_{ij} v \partial_i \omega_\varepsilon  \big]
+ \sum_{i,j=1}^d \partial_j \big[ (a_{ij}(x_0)-a_{ij}) \omega_{2\varepsilon}  \partial_i (\omega_\varepsilon v) \big] . 
$$
Since the left-hand side of \eqref{Eq-L-2} has constant coefficients, we can apply the result of Lemma \ref{RegPoiss} and obtain 
\begin{align*}
&\| \omega_\varepsilon v \|_{H^{1+\alpha}} \notag \\[8pt]
&\le
C\| g_{\varepsilon} \|_{H^{-1+\alpha}} \notag \\ 
&\le
C\| g \|_{H^{-1+\alpha}}
+ \sum_{i,j=1}^d 
\big( \| a_{ij} \, v \|_{H^\alpha}
+\| \partial_j a_{ij} \, v \|_{L^2} \big)
+ C\sum_{i,j=1}^d \| (a_{ij}(x_0)-a_{ij}) \omega_{2\varepsilon}   \partial_i (\omega_\varepsilon v)\|_{H^\alpha} \notag \\
&\le
C\| g \|_{H^{-1+\alpha}} + C\| v \|_{H^1}
+ C\sum_{i,j=1}^d \|(a_{ij}(x_0)-a_{ij}) \omega_{2\varepsilon} \|_{W^{1,d+\beta/2}} 
\|\omega_\varepsilon v\|_{H^{1+\alpha}} 
\,\,\,\mbox{(see Lemma \ref{Lemma:multiplier-H-alpha})} .
\end{align*} 
By using properties \eqref{properties-omega}--\eqref{aijx-aijx0} and H\"older's inequality, we have 
\begin{align*}
&\|(a_{ij}(x_0)-a_{ij}) \omega_{2\varepsilon} \|_{W^{1,d+\beta/2}} \\
&\le
\|(a_{ij}(x_0)-a_{ij}) \omega_{2\varepsilon} \|_{L^{d+\beta/2}}
+
\|\nabla a_{ij} \omega_{2\varepsilon} \|_{L^{d+\beta/2}}
+
\| (a_{ij}(x_0)-a_{ij}) \nabla \omega_{2\varepsilon} \|_{L^{d+\beta/2}} \\
&\le
C\varepsilon^{1-\frac{d}{d+\beta}} \varepsilon^{\frac{d}{d+\beta/2}} 
+
\|\nabla a_{ij} \|_{L^{d+\beta}} \|\omega_{2\varepsilon}\|_{L^s}
+
C\varepsilon^{1-\frac{d}{d+\beta}}  \varepsilon^{-1} \varepsilon ^{\frac{d}{d+\beta/2}}
\quad\mbox{with $\frac{1}{s}=\frac{1}{d+\beta/2}-\frac{1}{d+\beta}$}, \\
&\le
C\varepsilon^{\frac{1}{d+\beta/2}-\frac{1}{d+\beta}} .
\end{align*}
Combining the last two estimates, we obtain
\begin{align}\label{omega-v-H1+alpha}
\| \omega_\varepsilon v \|_{H^{1+\alpha}}
&\le
C\| g \|_{H^{-1+\alpha}}
+ C\| v \|_{H^1} + C
\varepsilon^{\frac{d}{d+\beta/2}-\frac{d}{d+\beta}} 
\|\omega_\varepsilon v\|_{H^{1+\alpha}} .
\end{align} 
By choosing a sufficiently small $\varepsilon$, the last term above can be absorbed by the left-hand side. Hence, we obtain
\begin{align*}
\| \omega_\varepsilon v \|_{H^{1+\alpha}} 
&\le
C\| g \|_{H^{-1+\alpha}}
+ C\| v \|_{H^1} .
\end{align*} 
Since $\| v \|_{H^1}\le C\| g \|_{H^{-1}}\le \| g \|_{H^{-1+\alpha}}$ (standard $H^1$ estimate of elliptic equations), it follows that 
\begin{align*}
\| \omega_\varepsilon v \|_{H^{1+\alpha}} 
&\le
C\| g \|_{H^{-1+\alpha}} .
\end{align*} 
Since this estimate holds in a neighborhood $B_{\varepsilon}(x_0)\cap \varOmega$ of every point $x_0\in\varOmega$ with a constant radius $\varepsilon$ (depending on $\|a_{ij}\|_{W^{1,d+\beta}}$), it follows that 
\begin{align*}
\|v \|_{H^{1+\alpha}} 
\le
C\| g \|_{H^{-1+\alpha}} . 
\end{align*} 
This proves Lemma \ref{RegPoiss}. 
In this proof, we have used the following lemma. 
\end{proof}

\begin{lemma}\label{Lemma:multiplier-H-alpha}
If $w\in W^{1,d+\beta/2}$, then the following inequality holds: 
\begin{align}\label{wv-H-alpha}
 \| w v\|_{H^{\alpha}} 
\le
C \| w \|_{W^{1,d+\beta/2}} 
\|v\|_{H^{\alpha}} \quad\forall\, v\in H^\alpha\,\,\,\mbox{and}\,\,\,\alpha\in[0,1]. 
\end{align}
\end{lemma}
\begin{proof}
Note that 
\begin{align*}
&
\| w v\|_{L^2} 
\le
C \| w \|_{L^\infty} \| v\|_{L^2} 
\le
C \| w \|_{W^{1,d+\beta/2}} \| v \|_{L^2} 
\quad\mbox{because $W^{1,d+\beta/2}\hookrightarrow L^\infty$},
\end{align*}
and
\begin{align*}
\| \nabla(w v)\|_{L^2} 
&\le
C ( \| w  \nabla v\|_{L^2} 
+ \| \nabla w \, v\|_{L^2}  ) \\
&\le
C ( \| w \|_{L^\infty} \| \nabla v\|_{L^2} 
+ \| \nabla w \|_{L^{d+\beta/2}} \| v\|_{L^s}  ) \quad\mbox{with}\,\,\, s=\frac{1}{\frac12-\frac{1}{d+\beta/2} }<\frac{2d}{d-2}\\
&\le
C \| w \|_{W^{1,d+\beta/2}} \| v\|_{H^1} .
\end{align*}
where the last inequality is due to the Sobolev embedding $H^1\hookrightarrow L^s$ for $1\le s<2d/(d-2)$. 
Therefore, \eqref{wv-H-alpha} holds for both $\alpha=0$ and $\alpha=1$. Since multiplying $v$ by $w$ is a linear operator on $v$, bounded on both $L^2$ and $H^1$, by the complex interpolation method this operator must also be bounded on $H^{\alpha}$ for $\alpha\in[0,1]$. This proves Lemma \ref{Lemma:multiplier-H-alpha}. 
\end{proof}
\medskip

The proof of Theorem \ref{Resolvent-Laplacian} is complete.



\def\baselinestretch{0.95}
{\small
\bibliographystyle{abbrv}
\bibliography{max-reg}

}
\end{document}